\documentclass[11pt, a4paper]{article}

\usepackage{amsmath, amssymb, amsthm, mathrsfs}
\usepackage{xcolor}
\usepackage{graphicx}
\usepackage{tikz}
\usetikzlibrary{calc}
\usepackage{booktabs, caption}
\usepackage{enumerate}
\usepackage{authblk}
\usepackage[noadjust]{cite}
\usepackage{CJK}
\usepackage[normalem]{ulem}
\usepackage{indentfirst}

\usepackage[colorlinks, linkcolor=blue, anchorcolor=green, citecolor=magenta]{hyperref}
\usepackage{cleveref}

\newtheoremstyle{mystyle}
{3pt}
{3pt}
{\itshape}
{}
{\bfseries}
{.}
{0.5em}
{\thmname{#1}\thmnumber{ #2}\thmnote{ {\normalfont(#3)}}}

\theoremstyle{mystyle}
\newtheorem{Theorem}{Theorem}[section]

\newtheorem{Lemma}[Theorem]{Lemma}

\newtheorem{Corollary}[Theorem]{Corollary}
\newtheorem{Proposition}[Theorem]{Proposition}

\theoremstyle{remark}

\theoremstyle{definition}
\newtheorem{algorithm}{Algorithm}
\theoremstyle{plain}

\numberwithin{equation}{section}

\begin{document}

\setcounter{page}{1}
\renewcommand{\thefootnote}{}
\newcommand{\remark}{\vspace{2ex}\noindent{\bf Remark.\quad}}
\renewcommand{\abovewithdelims}[2]{%
	\genfrac{[}{]}{0pt}{}{#1}{#2}}


\def\qed{\hfill$\Box$\vspace{11pt}}

\title {\bf   The complete nontrivial-intersection theorem
	for vector spaces}

\author{Dehai Liu\thanks{E-mail: \texttt{liudehai@mail.bnu.edu.cn}}\   \textsuperscript{a}}
\author{Benjian Lv\thanks{E-mail: \texttt{bjlv@bnu.edu.cn}}\   \textsuperscript{a}}
\author{Kaishun Wang\thanks{ E-mail: \texttt{wangks@bnu.edu.cn}}\ \textsuperscript{a}}
\author{Jie Wen\thanks{Corresponding author. E-mail: \texttt{jwen@mail.bnu.edu.cn}}\ \textsuperscript{a}}
\affil{ \textsuperscript{a} Laboratory of Mathematics and Complex Systems (Ministry of Education), School of
	Mathematical Sciences, Beijing Normal University, Beijing 100875, China}

\openup 0.5\jot
\date{}

\maketitle

\begin{abstract}
	
	 	A family of $k$-subspaces of an $n$-dimensional vector space over a finite field is $t$-intersecting if any two members intersect in a subspace of dimension at least $t$. A $t$-intersecting
	 	family is nontrivial if no $t$-subspace is contained in all its members. In this paper, we determine the largest nontrivial $t$-intersecting families for $k\geq t+2$ and $n\geq 2k+1$. As a result,  the complete nontrivial-intersection theorem for vector spaces is established.

	\vspace{2mm}
	
	\
	
	
\end{abstract}

\section{Introduction}

The celebrated Erd\H{o}s--Ko--Rado theorem
\cite{Erdos-Ko-Rado-1961} initiated the study of intersection problems
for finite sets. For sufficiently large $n$,
it determines the largest $t$-intersecting families of
$k$-subsets of an $n$-element set, where \emph{$t$-intersecting}
means that any two members have at least $t$ common elements.
Frankl \cite{Frankl-1976} and Wilson \cite{Wilson-1984} determined the
sharp threshold for this result. The problem for all parameters was
settled by Ahlswede and Khachatrian \cite{Ahlswede-Khachatrian-1997}
in their famous complete intersection theorem.

Another longstanding and active direction of research concerns the problem of  determining the largest $t$-intersecting families  with
fewer than $t$ elements common to all members. Such a family is said to be \emph{nontrivial}. 
The classical Hilton--Milner theorem \cite{Hilton-Milner-1967}
answered this question for $t=1$. A significant step was taken by
Frankl \cite{Frankl-1978}, who determined the largest nontrivial
$t$-intersecting families for $t\geq1$ and sufficiently large $n$.
Ahlswede and Khachatrian \cite{Ahlswede-Khachatrian-1996} subsequently
settled this problem in full with their complete nontrivial-intersection
theorem.

Intersection problems form a central part of extremal combinatorics.
They arise naturally in many combinatorial structures and have led
to many important results
\cite{Frankl-Tokushige-2016}.
Finite vector spaces provide a particularly well-studied setting, in which
many classical extremal problems for sets have natural analogues. Let $V$ be an $n$-dimensional vector space over the finite field $\mathbb F_q$, and write
${V\brack k}$ for the family of all $k$-subspaces of $V$.   A family $\mathcal F\subseteq{V\brack k}$ is \emph{$t$-intersecting} if
$\dim(F\cap G)\ge t$ for all $F,G\in\mathcal F$.   For $n\leq 2k-t$, the whole family
${V\brack k}$ is $t$-intersecting. We therefore turn to
$n\geq 2k-t+1$, and  present the complete Erd\H{o}s--Ko--Rado theorem for vector spaces.

\begin{Theorem}[{\cite{Hsieh-1975,Frankl-Wilson-1986,Tanaka-2006}}]
	\label{ekr}
	Let $k\geq t+1$ and $n\geq 2k-t+1$  be positive integers, and $V$ be an $n$-dimensional vector space over a finite field. Suppose that 
	$\mathcal F\subseteq{V\brack k}$ is a $t$-intersecting family of maximum size. The following hold.
	\begin{itemize}
		\item[\normalfont(i)] If $ n\leq 2k-1$, then $\mathcal{F}={Y\brack k}$ for some $Y\in{V\brack 2k-t}$.
		\item[\normalfont(ii)] If $n\geq 2k$, then  $\mathcal{F}$ consists of all $k$-subspaces
		containing a fixed $t$-subspace, or $n=2k$ and $\mathcal{F}={Y\brack k}$ for some $Y\in{V\brack 2k-t}$.
	\end{itemize}
\end{Theorem}

A $t$-intersecting family $\mathcal{F}\subseteq {V\brack k}$ is called \emph{trivial} if all its
members contain a common $t$-subspace, and \emph{nontrivial} otherwise.
For $n\geq 2k+1$, all extremal families in
Theorem \ref{ekr} are trivial. In analogy with the
Hilton--Milner theorem for sets, it is natural to ask for the structure of maximum-sized nontrivial $t$-intersecting families.
For $2k-t+1\leq n\leq 2k$, the theorem above  already shows that these
families consist of all $k$-subspaces of a fixed
$(2k-t)$-subspace. Moreover, such families are precisely the maximal
nontrivial $t$-intersecting families when $k=t+1$; see
\cite[Remark (ii), Section 9.3]{BCN-1989}. Hence, the natural range for this problem is $k\geq t+2$ and 
$n\geq 2k+1$.

Blokhuis et al. \cite{Blokhuis-etal-2010} proved a vector space version of the Hilton--Milner theorem, determining the largest nontrivial $1$-intersecting families for $n\geq2k+1$ and $(n,q)\neq(2k+1,2)$. For general $t$, an advance was made by Cao, Lv, Wang and Zhou
\cite{Cao-Lv-Wang-Zhou-2022}, who classified large maximal
nontrivial $t$-intersecting families and determined all families
of maximum size for $n\ge2k+t+\min\{4,2t\}$. Similar results were independently obtained by D'haeseleer \cite{Dhaeseleer-2023}. In 2023, Wang, Xu and Zhang \cite{Wang-Xu-Zhang-2023} established a vector space version of the Kruskal--Katona theorem. As an application, they completed the case $t=1$ for $n\ge2k+1$. The current best bound \cite{WXY-2024-220, Wang-Yang-2024},
obtained using the approach in \cite{Cao-Lv-Wang-Zhou-2022},
is $n\geq2k+1$ with $(n,q)\neq(2k+1,2)$.

The  remaining case
$t\geq 2$, $n=2k+1$ and $q=2$ is challenging. In this paper, we determine the largest nontrivial $t$-intersecting families   for all $k\geq t+2$ and $n\geq 2k+1$, with a unified proof. We note that  considering only the remaining case does not substantially simplify  our proof.   Before presenting our main result, we introduce the extremal nontrivial $t$-intersecting families.
Let $Z$ be a $(t+2)$-subspace, and let $X\subseteq M$ be a $t$-subspace and a $(k+1)$-subspace, respectively. Define
\begin{align*}
	\mathcal A(Z)&=
	\left\{F\in{V\brack k}:\dim(F\cap Z)\ge t+1\right\},\\
	\mathcal H(X,M)&=
	\left\{F\in{V\brack k}:X\subseteq F,\ \dim(F\cap M)\ge t+1\right\}
	\cup{M\brack k}.
\end{align*}

\begin{Theorem}\label{main}
	Let $k\geq t+2$ and $n\geq 2k+1$  be positive integers, and $V$ be an $n$-dimensional vector space over a finite field. Suppose that $\mathcal{F}\subseteq {V\brack k}$ is a nontrivial $t$-intersecting family of maximum size. The following hold.
	\begin{itemize}
		\item[\normalfont(i)] If  $k\geq 2t+2$, then $\mathcal{F}=\mathcal{H}(X, M)$ for some $M\in{V\brack k+1}$ and $X\in{M\brack t}$.  
		\item[\normalfont(ii)]   If  $k\leq 2t+1$, then $\mathcal{F}=\mathcal{A}(Z)$ for some $Z\in{V\brack t+2}$, or $(k,t)=(3,1)$ and $\mathcal{F}=\mathcal{H}(X, M)$ for some $M\in{V\brack 4}$ and $X\in{M\brack 1}$.
	\end{itemize}
\end{Theorem}

The rest of this paper is organized as follows. In Section \ref{2609123}, we introduce $t$-covers and  peeling 
procedure, and collect several important properties used throughout this paper.   Theorem \ref{main} is  proved in Section \ref{s3}.   For readability,  a number of technical estimates are proved in Appendix \ref{s4}.

\section{$t$-covers and peeling procedure} \label{2609123}

Recall that for integers $a \geq b \geq 0$, the Gaussian binomial coefficient is defined as
$$
{a\brack b}_q
=
\prod_{i=0}^{b-1}\frac{q^{a-i}-1}{q^{b-i}-1}.
$$
By convention, set ${a\brack b}_{q}=0$ if $b<0$ or $a<b$. In what follows,  write ${a\brack b}$ for short.

Let $V$ be an $n$-dimensional vector space over the finite field $\mathbb F_q$.
For  families $\mathcal{F}$ and $\mathcal{S}$ of subspaces of $V$, and for $S\leqslant V$, write 
\begin{align*}
\mathcal{F}[S]&=\left\{F\in\mathcal{F}:S\subseteq F\right\},\\
\mathcal{F}[\mathcal{S}]&=\left\{F\in\mathcal{F}: S\subseteq F\textnormal{ for some $S\in\mathcal{S}$}\right\}.
\end{align*}
For $i\geq 0$, the \emph{maximum $i$-degree} of $\mathcal{F}$ is defined to be
\begin{equation*}
\Delta_i(\mathcal{F})=\max\{|\mathcal{F}[S]|:S\leqslant V,\;\dim(S)=i\}.
\end{equation*}

 	Suppose that $\mathcal{F}$ is a family  of subspaces of dimension at least $t$. A subspace $T$ of $V$ is called a \textit{$t$-cover} of $\mathcal F$ if $\dim(T\cap F)\geq t$ for all $F\in\mathcal{F}$. If such a subspace exists, then the \textit{$t$-covering number} $\tau_t(\mathcal F)$ of $\mathcal{F}$ is the minimum  dimension of a $t$-cover. A $t$-cover is said to be 
 	\textit{minimal} if no proper subspace of it is a $t$-cover. Write    
 	the set of $\tau_{t}(\mathcal{F})$-dimensional $t$-covers of $\mathcal{F}$ as $\mathcal{T}_{t}(\mathcal{F})$.

The subsequent  three results  concern degrees
and will be used throughout the paper.

\begin{Proposition}\label{degree}
	Let $n\geq k\geq t\geq 1$. Suppose that   $\mathcal{F}$ is a  family of subspaces of $V$. For each $s\geq 0$, if there exist
	a $t$-cover $G\in{V\brack k}$ of $\mathcal{F}$ and
	a subspace $S\in{V\brack s}$ with
	$|\mathcal{F}[S]|=\Delta_s(\mathcal{F})$  such that
	$\dim(S\cap G)<t$, then
  $\Delta_{s}(\mathcal{F})\leq \Delta_{s+1}(\mathcal{F}) {k-t+1\brack 1}$. 
\end{Proposition}

\begin{proof}
 Choose $A\in{G\brack k-t+1}$ with  $\dim(A\cap S)=0$. 
	For each $F\in\mathcal{F}$, we have
	$$
	\dim(F\cap A)=\dim(F\cap G\cap A)
	\geq\dim(F\cap G)+\dim(A)-\dim(G)\geq1.
	$$
	Therefore, we have $\mathcal{F}[S]\subseteq \bigcup_{L\in{A\brack 1}}\mathcal{F}[S+L]$. 
	This together with $|\mathcal{F}[S+L]|\leq \Delta_{s+1}(\mathcal{F})$ for any $L\in{A\brack 1}$ yields the desired result. 
\end{proof}

From now on, we assume that $t$-intersecting families are non-empty, since those families we are concerned with  must be non-empty.

\begin{Corollary}\label{key}
	Let $n\geq k\geq t\geq 1$. Suppose that $\mathcal{F}\subseteq {V\brack k}$ is a  $t$-intersecting family. Set $\tau=\tau_{t}(\mathcal{F})$. Then $$\Delta_{s}(\mathcal{F})\leq \Delta_{\tau-1}(\mathcal{F}) {k-t+1\brack 1}^{\tau-s-1} \textnormal{ for any } 0\leq s< \tau.$$ 
	Moreover, we have 
	$$\Delta_{\tau-1}(\mathcal{F})\leq {n-\tau+1\brack k-\tau+1} -q^{(k-t+1)(k-\tau+1)} {n-k-\tau+t\brack k-\tau+1}\leq {k-t+1\brack 1}{n-\tau\brack k-\tau}.$$
\end{Corollary}
\begin{proof}
	For each $0\leq s<\tau$ and $S\in{V\brack s}$, since $S$ is not a $t$-cover of $\mathcal{F}$, there exists $G\in\mathcal{F}$ such that $\dim(G\cap S)<t$. The first inequality follows by applying  Proposition \ref{degree} repeatedly.  
	
	Pick  $S\in{V\brack \tau-1}$ and
	$G\in\mathcal{F}$ with $\dim(S\cap G)<t$.
	Choose $A\in{G\brack k-t+1}$ with $\dim(A\cap S)=0$.  Then by $\dim(F\cap A)=\dim(F\cap G\cap A)\geq 1$ we have
	$$\mathcal{F}[S]\subseteq\left\{F\in{V\brack k}:S\subseteq F,\ \dim(F\cap A)\geq1\right\}.$$
The family displayed on the right-hand side has size
	${n-\tau+1\brack k-\tau+1} -q^{(k-t+1)(k-\tau+1)} {n-k-\tau+t\brack k-\tau+1}$.
	On the other hand, every subspace counted contains $S+L$ for some
	$L\in{A\brack1}$, so there are at most 
	${k-t+1\brack1}{n-\tau\brack k-\tau}$ such subspaces. This proves the upper bounds on $\Delta_{\tau-1}(\mathcal{F})$.
\end{proof}

\begin{Lemma}\label{2608181}
	Let  $\mathcal{F}$ be a family of subspaces of $V$. Suppose that $A$ and $B$ are  two $t$-covers of $\mathcal{F}$ with $\dim(A\cap B)\leq t$. Set $c=\dim(A\cap B)$. Then 
	$$\left|\mathcal{F}\right|\leq \sum_{j=0}^{c}q^{2j(t-c+j)}{c\brack j}{\dim(A)-c\brack t-c+j}{\dim(B)-c\brack t-c+j}\Delta_{2t-c+j}(\mathcal{F}).$$
\end{Lemma}	
\begin{proof}
	Set   $C=A\cap B$ and  
	$\mathcal{C}_{j}=\left\{F\in\mathcal{F}: \dim(F\cap C)=j\right\}$ for $0\leq j\leq c$. Then $\mathcal{F}=\cup_{j=0}^{c}\mathcal{C}_{j}$. For a subspace $H$ of $C$, write 
	$$\mathcal{D}(H)=\left\{(X, Y)\in{A\brack t}\times {B\brack t}: X\cap C=Y\cap C=H\right\}, $$ 
	and we have
	\begin{equation}\label{2609241}
		|\mathcal{D}(H)|\leq q^{2(c-\dim(H))(t-\dim(H))}{\dim(A)-c\brack t-\dim(H)}{\dim(B)-c\brack t-\dim(H)}.
	\end{equation}
For $(X,Y)\in\mathcal D(H)$, by $H\subseteq X\cap Y= X\cap Y\cap C\subseteq H$, we have $X\cap Y=H$ and $\dim(X+Y)=2t-\dim(H)$.

 Recall that $\mathcal{F}=\cup_{j=0}^{c}\mathcal{C}_{j}$.	It is sufficient to show 
	\begin{equation}\label{2609242}
		|\mathcal{C}_{c-j}|\leq q^{2j(t-c+j)}{c\brack j}{\dim(A)-c\brack t-c+j}{\dim(B)-c\brack t-c+j}\Delta_{2t-c+j}(\mathcal{F})\;\;\mbox{for}\;\;0\leq j\leq c.
	\end{equation}
	Fix such a $j$. If $\mathcal{C}_{c-j}$  is empty, then the desired result is clear. Next, we may assume that $\mathcal{C}_{c-j}$ is non-empty. Choose $F\in \mathcal{C}_{c-j}$. Since  both $A$ and $B$ are $t$-covers of $\mathcal{F}$, there exist $X\in{F\cap A\brack t}$ and $Y\in{F\cap B\brack t}$ such that $F\cap C\subseteq X$ and $F\cap C\subseteq Y$. Note that 
	$X\cap C\subseteq F\cap C$ and $Y\cap C\subseteq F\cap C$. Then $X\cap C=Y\cap C=F\cap C$. It follows that 
	$$\mathcal{C}_{c-j}\subseteq \bigcup_{H\in{C\brack c-j}}\left( \bigcup_{(X, Y)\in \mathcal{D}(H)} \mathcal{F}[X+Y]\right).$$
	Then (\ref{2609242}) follows 	from  (\ref{2609241}). This completes the proof.
\end{proof}

Denote the family of all subspaces of $V$ with dimension at most  $k$ by ${V\brack \leq k}$.  Suppose that $\mathcal{F}\subseteq{V\brack \leq k}$ 
satisfies $\dim(F\cap G)\geq t$
for any $F,G\in\mathcal{F}$. Such a family is also called \textit{$t$-intersecting}. A family $\mathcal{F}^{*}$ is called a \textit{fingerprint} of $\mathcal{F}$ if the following hold.
\begin{itemize}
	\item[\normalfont(i)] Every member of $\mathcal{F}^{*}$ is contained in some member of $\mathcal{F}$;
	\item[\normalfont(ii)] Every member of $\mathcal{F}$ contains some member of $\mathcal{F}^{*}$;
	\item[\normalfont(iii)] Every member of $\mathcal{F}^{*}$ is a minimal $t$-cover of $\mathcal{F}^{*}$.
\end{itemize}

	\begin{Lemma}
		Let $\mathcal{F}\subseteq {V\brack \leq k}$ be a $t$-intersecting family. Then there exists a fingerprint  of $\mathcal{F}$.
	\end{Lemma}
	\begin{proof}
		Let $\mathscr{C}$ be the collection of all $t$-intersecting families satisfying conditions (i) and (ii) in the definition of a fingerprint of $\mathcal{F}$. Since $\mathcal{F}\in\mathscr{C}$, this collection is non-empty. Choose $\mathcal F^*\in\mathscr C$ minimizing
		$\sum_{F\in\mathcal F^*}\dim(F)$. Let us verify that $\mathcal{F}^*$ is a fingerprint of $\mathcal{F}$. It suffices to verify the condition (iii). Suppose for contradiction that $F\in\mathcal{F}^{*}$ is not a minimal $t$-cover of $\mathcal{F}^{*}$. Since $\mathcal{F}^{*}$ is $t$-intersecting, a proper subspace $G$ of $F$ is a minimal $t$-cover of $\mathcal{F}^{*}$. Then the $t$-intersecting family $(\mathcal{F}^{*}\setminus\{F\})\cup\{G\}$ still satisfies (i), (ii) and it has a strictly smaller sum of dimensions, contradicting the choice of $\mathcal F^*$.
	\end{proof}
	
Our approach is based on the peeling procedure and the $t$-cover method. The peeling procedure	was introduced by Kupavskii and Zakharov \cite{Kupavskii-Zakharov-2024} and  combined with the $t$-cover method by Wen and Lv \cite{Wen-Lv-largefamilyI}. We also draw inspiration from the $q$-analog of the
peeling procedure developed by Ihringer and Kupavskii
\cite{IK-2026-02698}.

The recently work \cite{Li-Lv-Wen-2026} also uses the $t$-cover method without employing the peeling procedure. Moreover, the extremal problems it considers and the extremal configurations it characterizes differ from those in the present paper.
	
We present the peeling procedure in the form of the following algorithm.

\begin{algorithm}\label{peeling}
	Suppose that $\mathcal{F}\subseteq {V\brack \leq k}$ is a  $t$-intersecting family,
	and a fingerprint $\mathcal{S}_{0}$ of $\mathcal{F}$ is given. Set $N=0$. For
	$i=0,1,\ldots,k-t$, do the following:
	\begin{itemize}
		\item[(a)] Set $\mathcal{X}_{i}=\mathcal{S}_{i}\cap {V\brack k-i}$.
		If $\mathcal{S}_{i}=\mathcal{X}_{i}$, then terminate the procedure and set $N=i$.
		\item[(b)] Find a fingerprint of $\mathcal{S}_{i}\setminus\mathcal{X}_{i}$,
		and denote it by $\mathcal{S}_{i+1}$.
	\end{itemize}
	Output $N$ and $(\mathcal{S}_{i},\mathcal{X}_{i})$ for $0\leq i\leq N$.
\end{algorithm}

Let us prove several basic properties of the output families.

\begin{Lemma}\label{2608182}
	The following hold for $i=0,1,\ldots, N$.
\begin{itemize}
	\item[\normalfont(i)] $\mathcal{S}_{i}\subseteq {V\brack \leq k-i}$ and $\mathcal{X}_{i}\subseteq {V\brack k-i}$. Moreover, both of them are  antichains. 
	\item[\normalfont(ii)] $\mathcal{F}[\mathcal{S}_{i-1}]\subseteq \mathcal{F}[\mathcal{X}_{i-1}]\cup\mathcal{F}[\mathcal{S}_{i}]$ and $\mathcal{F}=\left(\bigcup_{j=0}^{i-1}\mathcal{F}[\mathcal{X}_{j}]\right)\cup \mathcal{F}[\mathcal{S}_{i}]$ for $i\geq 1$.
	\item[\normalfont(iii)]  If  $\mathcal{H}\subseteq \mathcal{S}_{i}$ consists of $a$-subspaces, then 
	 $$\Delta_{s}(\mathcal{H})\leq {k-t-i+1\brack 1}^{a-s}\ \ \textnormal{for any $0\leq s\leq a$.}$$
\end{itemize}
\end{Lemma}
\begin{proof}
	(i) The facts that $\mathcal{S}_{i}\subseteq {V\brack \leq k-i}$ and $\mathcal{X}_{i}\subseteq {V\brack k-i}$ are immediate. Since each member of $\mathcal{S}_{i}$ is a minimal $t$-cover of $\mathcal{S}_{i}$, we know $\mathcal{S}_{i}$ is an antichain, and then so is $\mathcal{X}_i$.
	
	(ii) For $i<N$, each member of $\mathcal S_i\setminus\mathcal X_i$
	contains a member of $\mathcal S_{i+1}$ by property (ii) in the
	definition of a fingerprint. Then
	$\mathcal F[\mathcal S_i]\subseteq
	\mathcal F[\mathcal X_i]\cup\mathcal F[\mathcal S_{i+1}]$.
	Iterating this inclusion and using
	$\mathcal F=\mathcal F[\mathcal S_0]$ gives the second equality.
	
	(iii) If $a=s$ or $\Delta_{s}(\mathcal{H})=0$, then the inequality is clear.  We may assume that  $a>s$ and $\Delta_{s}(\mathcal{H})>0$. Choose an $s$-subspace $S$ with $|\mathcal H[S]|=\Delta_s(\mathcal H)$, and pick $F\in\mathcal H[S]$. Since $S\subsetneq F$ and $F$ is a minimal $t$-cover of $\mathcal S_i$, there exists $G\in\mathcal S_i$ with $\dim(S\cap G)<t$. Note that $G$ is a $t$-cover of $\mathcal{H}$ and $\dim(G)\leq k-i$. By Proposition \ref{degree}, we have
	$
	\Delta_s(\mathcal H)
	\leq
	\Delta_{s+1}(\mathcal H)
{k-t-i+1\brack 1}$.
	Applying this inequality repeatedly and using $\Delta_a(\mathcal{H})=1$ yields the desired result. 
\end{proof}
To state the bounds on the sizes of the families $\mathcal{X}_{i}$
concisely, write 
\begin{align}
	\psi(k,t,a,i)	&={a\brack t}{k-t-i+1\brack1}^{k-t-i},
	\label{2609111}\\
	\varphi(k,t,a,b,i)
	&=\sum_{j=0}^{t}q^{2j^{2}}
	{t\brack j}{a-t\brack j}{b-t\brack j}
	{k-t-i+1\brack1}^{k-t-i-j}.
	\label{2609112}
\end{align}

\begin{Lemma}\label{2609061}
	Set $m_{i}=\min_{F\in\mathcal{S}_{i}}\dim(F)$ for $0\leq i\leq N$.	Then  $ m_{0}\geq m_{1}\geq \cdots \geq m_{N}\geq t$. Moreover, the following hold.
	\begin{itemize}
		\item[\normalfont(i)] $\left|\mathcal{X}_{i}\right|\leq \varphi(k,t,m_{i},k-i,i)\leq \psi(k,t,m_{i}, i)$ for  $0\leq i\leq N$.
		\item [\normalfont(ii)] If $N\geq 1$, then $m_{i}\leq k-i-1$ for $0\leq i\leq N-1$. 
		\item[\normalfont(iii)]  If  $\mathcal{T}_{t}(\mathcal{F})\cap\mathcal{S}_{0}$ is non-empty, then  $m_{i}\leq \tau_{t}(\mathcal{F})$ for  $0\leq i\leq N$. 
	\end{itemize}
\end{Lemma}
\begin{proof}
	
	For $i<N$, choose $F_i\in\mathcal S_i\setminus\mathcal X_i$
	with $\dim(F_i)=m_i$.
	By property (ii) in the definition of a fingerprint, $F_i$
	contains a member of $\mathcal S_{i+1}$. Hence $m_{i+1}\le m_i$ for $i<N$, and thus
	$m_0\ge m_1\ge\cdots\ge m_N\ge t$.
	
	(i)  Pick $F_{i}\in \mathcal{S}_{i}$ with $\dim(F_{i})=m_{i}$. 
	Since $F_{i}$ is a minimal $t$-cover of $\mathcal{S}_{i}$, there exists $G_{i}\in \mathcal{S}_{i}$ such that $\dim(F_{i}\cap G_{i})=t$.  By Lemma \ref{2608181}, Lemma \ref{2608182} (i) and (iii), we have $\left|\mathcal{X}_{i}\right|\leq \varphi(k,t, m_{i},\dim(G_{i}),i)\leq \varphi(k,t,m_{i},k-i,i)$. For $0\le j\le k-t-i$, we have
	$$q^{j^{2}}{k-t-i\brack j}\leq q^{j^{2}}{k-t-i-j+1\brack 1}^{j}\leq {k-t-i+1\brack 1}^{j}.$$
Setting $(a,b,c)=(m_{i}-t, t,t)$ in Lemma \ref{2609151}, we have ${m_{i}\brack t}=\sum_{j=0}^{t}q^{j^{2}}{m_{i}-t\brack j}{t\brack j}$. It follows that $\varphi(k,t,m_i,k-i,i)\le\psi(k,t,m_i,i)$.
	
	(ii) This follows directly from $\emptyset\neq\mathcal S_i\setminus\mathcal X_i\subseteq{V\brack\leq k-i-1}$.
	
	(iii) As $\mathcal T_t(\mathcal F)\cap \mathcal S_0$ is non-empty,
 we get $m_0\le\tau_t(\mathcal F)$. Then $m_i\le \tau_t(\mathcal F)$ as $m_i\leq m_0$.
\end{proof}
\begin{Lemma}\label{2609056}
	If $N\geq 1$, then $\mathcal S_N=\mathcal S_{N-1}\setminus\mathcal X_{N-1}$ and $\left|\mathcal{X}_{N-1}\right|\leq \varphi(k,t,k-N,k-N,N-1)$.
\end{Lemma}
\begin{proof}
	Every member of $\mathcal S_{N-1}\setminus\mathcal X_{N-1}$
	has dimension at most $k-N$, while $\mathcal S_N=\mathcal{X}_N$ consists
	of $(k-N)$-subspaces. Then by the properties (i) and (ii) in the definition
	of a fingerprint, we get $\mathcal S_N=\mathcal S_{N-1}\setminus\mathcal X_{N-1}$. By the minimality of the members of $\mathcal S_N$, there are 
	$A_N,B_N\in\mathcal S_N$ with $\dim(A_N\cap B_N)=t$.
	Both of them are $t$-covers of $\mathcal X_{N-1}$, so
	Lemmas~\ref{2608181} and~\ref{2608182} (iii) give
	$|\mathcal X_{N-1}|\le\varphi(k,t,k-N,k-N,N-1)$. This completes the proof.
\end{proof}

We proceed by showing two lemmas which are used to estimate the size of $\mathcal{S}_{N}$.

\begin{Lemma}\label{2609221}
	Let $n\geq k\geq t\geq 1$.  Suppose that $\mathcal{F}\subseteq{V\brack k}$ is non-empty and each member of $\mathcal{F}$ is a minimal $t$-cover of $\mathcal{F}$. Pick $T\in {V\brack k-1}$ with $\left|\mathcal{F}[T]\right|=\Delta_{k-1}(\mathcal{F})$. Set
	$$M=\sum_{F\in \mathcal{F}[T]}F, \ \ a=\dim(M)-k+1,\ \  \mathcal{G}=\left\{F\in\mathcal{F}: \dim(F\cap T)=t-1\right\}.$$
	\begin{itemize}
		\item[\normalfont(i)] $\mathcal{G}$ is non-empty.
		\item[\normalfont(ii)] $M\subseteq T+G$ and $\dim(G\cap M)=a+t-1$ for any $G\in\mathcal{G}$. In particular, we have 
		$$\left|\mathcal{G}\right|\leq q^{a(k-t)}{k-1\brack t-1}\Delta_{a+t-1}(\mathcal{F}).$$
		\item[\normalfont(iii)]  $\Delta_{s}(\mathcal{F})\leq {a\brack 1}{k-t+1\brack 1}^{k-s-1}$ for any $0\leq s\leq k-1$.
	\end{itemize}
\end{Lemma}
\begin{proof}
	(i) Pick $F\in \mathcal{F}[T]$.  Since $F$ is a minimal $t$-cover of $\mathcal{F}$, there exists $G\in\mathcal{F}$ such that $\dim(G\cap T)<t$. It follows from  $\dim(F\cap G)\geq t$ that $\dim(G\cap T)=t-1$.
	
	(ii) Let $G\in\mathcal{G}$. For each $F\in\mathcal{F}[T]$, by $\dim(F\cap G)\geq t$, we get $F=T+(F\cap G)\subseteq T+G$. This implies $M\subseteq T+G$. We further conclude  $T+G=M+G$. Then
	$$\dim(G\cap M)=k+\dim(M)-\dim(T+G)=a+t-1.$$ 
	It follows that $$\left|\mathcal{G}\right|\leq \sum_{H\in{M\brack a+t-1},\ \dim(H\cap T)=t-1}\left|\mathcal{F}[H]\right|\leq q^{a(k-t)}{k-1\brack t-1}\Delta_{a+t-1}(\mathcal{F}).$$
	This completes the proof of (ii).
	
	(iii) Let $0\leq s\leq k-1$ and $S\in{V\brack s}$ with $\left|\mathcal{F}[S]\right|=\Delta_{s}(\mathcal{F})$. Since  $\mathcal{F}$ is non-empty, we know $S\subsetneq F$ for some $F\in\mathcal{F}$. Recall that $F$ is a minimal $t$-cover. Then $\dim(S\cap G)<t$ for some $G\in\mathcal{F}$. It follows from Proposition \ref{degree} that $\Delta_{s}(\mathcal{F})\leq \Delta_{s+1}(\mathcal{F}){k-t+1\brack 1}$.  
	
	Repeating the above procedure, we have $\Delta_{s}(\mathcal{F})\leq \Delta_{k-1}(\mathcal{F}){k-t+1\brack 1}^{k-s-1}$. This together with $\Delta_{k-1}(\mathcal{F})\leq {a\brack 1}$ yields (iii).
\end{proof}

\begin{Lemma}\label{2609091}
	Let $n\geq t+2\geq 3$ and $\mathcal{F}\subseteq{V\brack t+2}$. If each member of $\mathcal{F}$ is a minimal $t$-cover of $\mathcal{F}$, then $\left|\mathcal{F}\right|\leq {t+4\brack 2}$.  Moreover,  if equality holds, then $\mathcal{F}={M\brack t+2}$ for some $M\in{V\brack t+4}$. 
\end{Lemma}	
\begin{proof}
	We may assume that $\mathcal{F}$ is non-empty.
	Let $T,M,a$ and $\mathcal G$ be as in Lemma \ref{2609221}
	with $k=t+2$.
	By parts (i) and (ii), we have $\mathcal G\neq\emptyset$ and
	$t+2\leq\dim(M)\leq t+4$, so $1\leq a\leq3$.


	Suppose  $a\leq 2$. 
	Set
	$
	\mathcal H=\mathcal F\setminus\mathcal G$. It follows from $\dim(F\cap T)\geq t-1$ for any $F\in\mathcal{F}$ that
	$\mathcal{H}=\{F\in\mathcal F:\dim(F\cap T)\geq t\}$.	
	Note that ${a\brack1}<q^a$ and ${3\brack1}<q^3$. By Lemma \ref{2609221} (ii) and (iii),
	we get
	$$
	|\mathcal G|
	\leq q^{2a}{t+1\brack2}{a\brack1}{3\brack1}^{2-a}
	\leq q^6{t+1\brack2}.
	$$
	Fix $G\in\mathcal G$. Then $G$ and $T$ are $t$-covers
	of $\mathcal H$, with $\dim(G\cap T)=t-1$.
	By Lemma \ref{2609221} (iii), we have
	$\Delta_{t+1}(\mathcal F)\leq{a\brack1}\leq{2\brack1}$.
	This together with $\Delta_{t+2}(\mathcal F)=1$
	and Lemma \ref{2608181} yields
	$$
	|\mathcal H|
	\leq {3\brack1}{2\brack1}^2
	+q^{4}{3\brack1}{t-1\brack1}.
	$$
	Combining these bounds with Lemma \ref{2608201}, we obtain inequality $|\mathcal{F}|<{t+4\brack 2}$ holds.
	
	Suppose $a=3$. Then $\dim(M)=t+4$. By Lemma \ref{2609221} (ii), we know  $M=T+G$ for any $G\in\mathcal{G}$. 
	Set $\mathcal{D}=\left\{F\in\mathcal{F}: F\nsubseteq M\right\}$. 
	If $\mathcal{D}$ is empty, then
	$\mathcal{F}\subseteq{M\brack t+2}$, yielding
	$|\mathcal F|\leq{t+4\brack2}$,
	with equality if and only if $\mathcal F={M\brack t+2}$.
	Next, assume  $\mathcal{D}$ is non-empty, and show $|\mathcal{F}|<{t+4\brack 2}$.

	Choose  $G\in\mathcal{G}$.
	For each $D\in\mathcal{D}$, the definition of $\mathcal{D}$ gives $\dim(D\cap T)=t$ and $\dim(D\cap M)\leq t+1$, which implies $$t+1\geq \dim(D\cap (T+G))\geq \dim(D\cap T)+\dim(D\cap G)-\dim(D\cap T\cap G)\geq t+1.$$ Therefore, we have 
	\begin{equation}\label{2609222}
		\dim(D\cap T)=\dim(D\cap G)=t,\ \ \dim(D\cap M)=t+1, \ \ T\cap G\subseteq D \textnormal{ for any $D\in\mathcal{D}$.}
	\end{equation}

	Pick $H\in {M\brack t+1}$ such that $\mathcal{D}[H]$ is non-empty. Since each member of $\mathcal{F}$ is a minimal $t$-cover, we know $H$ is not a $t$-cover of $\mathcal{F}$. Then there exists $F\in\mathcal{F}$ such that $\dim(F\cap  H)=t-1$. Such a subspace $F$ must lie in $\mathcal{D}$ as $\mathcal{F}[H]\cup\{F\}$ is $t$-intersecting. Moreover, we know $\dim(D\cap F)=t$ and  $D=H+(D\cap F)$ for any $D\in\mathcal{D}[H]$.  It follows from (\ref{2609222}) that 
	$$\mathcal{D}[H]\subseteq \left\{H+S: S\in {F\brack t},\ T\cap G \subseteq S,\ \dim(S\cap M)=t-1 \right\}.$$
	By (\ref{2609222}), we have
	$
	\mathcal {D}=\cup_{H} \mathcal{D}[H],
	$
	where $H$ ranges over all $(t+1)$-subspaces of $M$ satisfying
	$\dim(H\cap T)=\dim(H\cap G)=t$ and $T\cap G\subseteq H$.
	Consequently, we know $|\mathcal{D}|\leq q^{2}{2\brack 1}{3\brack 1}$.
	
	Pick $X\in\mathcal{D}$. Note that $X\cap M$ is not a $t$-cover of $\mathcal{F}$. Similarly, there exists $Y\in\mathcal{D}$ such that $\dim(Y\cap X\cap M)=t-1$, i.e., $\dim((X\cap M)\cap (Y\cap M))=t-1$.  Since $F\cap X=F\cap X\cap M$ and $F\cap Y=F\cap Y\cap M$ for any $F\in\mathcal{F}\setminus \mathcal{D}$, we get $X\cap M$ and $Y\cap M$ are $t$-covers of $\mathcal{F}\setminus \mathcal{D}$ with dimension $t+1$. Consequently, Lemma \ref{2608181} yields $|\mathcal{F}\setminus \mathcal{D}|\leq {2\brack 1}^{2} \Delta_{t+1}(\mathcal{F})+q^{4}{t-1\brack 1}\Delta_{t+2}(\mathcal{F})$. By $\Delta_{t+1}(\mathcal{F})\leq {3\brack 1}$ and $\Delta_{t+2}(\mathcal{F})=1$, we know 
	$$\left|\mathcal{F}\right|\leq q^{2}{2\brack 1}{3\brack 1}+ {2\brack 1}^{2} {3\brack 1}+q^{4}{t-1\brack 1}={2\brack 1}{3\brack 1}^{2}+q^{4}{t-1\brack 1}.$$
	It follows from Lemma \ref{2608201} that $|\mathcal{F}|<{t+4\brack 2}$. This finishes the proof.
\end{proof}

A key step in applying Algorithm \ref{peeling} is the choice of the
initial fingerprint. To this end, for a $t$-intersecting family $\mathcal{F}\subseteq {V\brack k}$,  write
$$\mathcal{M}_{t}(\mathcal{F})=\left\{T \leqslant V:\dim(T)\leq k,\ \textnormal{$T$ is a minimal $t$-cover of $\mathcal{F}$} \right\}.$$  

\begin{Lemma} Let $k\geq t\geq 1$ and $n\geq 2k-t+1$. Suppose that $\mathcal{F}\subseteq {V\brack k}$ is a maximal $t$-intersecting family. The following hold. 
	\begin{itemize} \item[\normalfont(i)] $\mathcal{M}_{t}(\mathcal{F})$ is $t$-intersecting. \item[\normalfont(ii)] $\mathcal{M}_{t}(\mathcal{F})$ is a fingerprint of $\mathcal{F}$. 
	\end{itemize} 
\end{Lemma}

\begin{proof}
	(i) Suppose for contradiction that $\dim(T_{1}\cap T_{2})\leq t-1$ for some $T_{1},T_{2}\in\mathcal{M}_{t}(\mathcal{F})$. Choose $A\in{T_{2}\brack \dim(T_{2})-t+1}$ with $\dim(A\cap T_{1})=0$. Pick a subspace $U$ of $V$ such that $V=A\oplus U$ and $T_{1}\subseteq U$. Since
	$
	\dim(U)=n-\dim(T_{2})+t-1\geq k
	$, 
	there exists $F\in{U\brack k}$ with $T_{1}\subseteq F$. Since $T_{1}$ is a $t$-cover of $\mathcal{F}$, the maximality of $\mathcal{F}$ yields $F\in\mathcal{F}$. However, by $\dim(F\cap A)=0$ and $A\subseteq T_{2}$, we have
	$$
	0=\dim(F\cap A)=\dim(F\cap T_{2})+\dim(A)-\dim((F\cap T_{2})+A)\geq \dim(F\cap T_{2})-t+1.
	$$
	This implies $\dim(F\cap T_{2})\leq t-1$, a contradiction. Hence $\mathcal{M}_{t}(\mathcal{F})$ is $t$-intersecting.
	
	(ii) For each $T\in\mathcal{M}_{t}(\mathcal{F})$, choose $F\in{V\brack k}$ with $T\subseteq F$. Since $T$ is a $t$-cover of $\mathcal{F}$, the maximality of $\mathcal{F}$ gives $F\in\mathcal{F}$. Conversely, each member of $\mathcal{F}$ is a $t$-cover of $\mathcal{F}$, and hence contains some member of $\mathcal{M}_{t}(\mathcal{F})$.
	
	It remains to show that each member of $\mathcal{M}_{t}(\mathcal{F})$ is a minimal $t$-cover of $\mathcal{M}_{t}(\mathcal{F})$. Pick $T\in\mathcal{M}_{t}(\mathcal{F})$. By part (i), $T$ is a $t$-cover of $\mathcal{M}_{t}(\mathcal{F})$. For any proper subspace $H$ of $T$, the minimality of $T$ gives some $G\in\mathcal{F}$ with $\dim(H\cap G)<t$. By the preceding paragraph, there exists $S\in\mathcal{M}_{t}(\mathcal{F})$ with $S\subseteq G$. Then
	$
	\dim(H\cap S)\leq\dim(H\cap G)<t,
	$
	so $H$ is not a $t$-cover of $\mathcal{M}_{t}(\mathcal{F})$. Therefore $T$ is a minimal $t$-cover of $\mathcal{M}_{t}(\mathcal{F})$, and  $\mathcal{M}_{t}(\mathcal{F})$ is a fingerprint of $\mathcal{F}$.
\end{proof}

\section{Proof of Theorem \ref{main}}\label{s3}

We first fix the notation for the rest of this section. 
  Let $k\geq t+2$ and $n\geq 2k+1$  be positive integers, and $V$ be an $n$-dimensional vector space over the finite field $\mathbb{F}_{q}$. Suppose that $\mathcal{F}\subseteq {V\brack k}$ is a maximal  $t$-intersecting family with $\tau:=\tau_{t}(\mathcal{F})\geq t+1$.   Apply  Algorithm \ref{peeling} to $\mathcal{F}$ with $\mathcal{S}_{0}=\mathcal{M}_{t}(\mathcal{F})$, and denote the output by $N$ and $(\mathcal{S}_{i},\mathcal{X}_{i})$ for $0\leq i\leq N$.

By Lemma \ref{2608182} (ii) and the fact that $\mathcal S_N=\mathcal X_N$, we have
\begin{equation}\label{equdecomp}
\mathcal F
=\left(\bigcup_{i=0}^{N-1}\mathcal F[\mathcal X_i]\right)
\cup\mathcal F[\mathcal S_N].
\end{equation}
We estimate separately the contributions of the \emph{early layers}
$\mathcal F[\mathcal X_i]$, $0\le i<N$, and the \emph{final layer}
$\mathcal F[\mathcal S_N]$.
To record the bounds for  early layers, we introduce
the following two functions.
\begin{align}
		\Phi_{1}(n,k,t,a,b,j)&:=\varphi(k,t,a,k,0)
	+\sum_{i=1}^{j}\psi(k,t,b,i){n-k+i\brack i},	\label{2609115}\\
	\Phi_{2}(n,k,t,a,j)&:=\varphi(k,t,a,k,0)
	+\sum_{i=1}^{j}\psi(k,t,k-i-1,i){n-k+i\brack i},
	\label{2609113}
\end{align}
where we  adopt the convention that  an empty sum is equal to $0$. Write $a(n,k,t)=|\mathcal A(Z)|$ and
$h(n,k,t)=|\mathcal H(X,M)|$; these sizes are independent
of the choice of the defining subspaces.  We further let 
$m(n,k,t)=\max\{a(n,k,t), h(n,k,t)\}$.

 Note that the stopping index $N$ satisfies $0\leq N\leq k-t$.
 The proof is divided into the following four cases:
 $N=0$, $1\leq N\leq k-t-2$, $N=k-t-1$ and $N=k-t$.

\subsection{$N=0$}

In this situation, the initial fingerprint consists entirely of
$k$-subspaces. In particular, $\tau=k$. 
We choose a $(k-1)$-subspace of maximum degree and split
$\mathcal F$ according to whether its members meet this
subspace in dimension $t-1$ or at least $t$.

\begin{Lemma}\label{2609122}
Suppose that
	$N=0$. Then
	$|\mathcal F|<m(n,k,t)$.
\end{Lemma}
\begin{proof}
	
	Suppose $k=t+2$. By Lemma \ref{2609091}, we have $\left|\mathcal{F}\right|\leq {t+4\brack 2}$. Note that 
	$$a(2t+5,t+2, t)-{t+4\brack 2}=\frac{q^{3}(q^{t+1}-1)(q^{t+2}-1)}{(q^{2}-1)(q-1)}>0.$$
	Since $n\geq 2t+5$ and $a(n,t+2,t)$ is increasing in \(n\), the desired result follows. Next, we assume $k\geq t+3$.
	
	Since $\tau=k$, every member of $\mathcal F$
	is a minimal $t$-cover of $\mathcal F$.
	Let $T,M,a$ and $\mathcal G$ be as in Lemma \ref{2609221}.
	By parts (i) and (ii), $\mathcal G\neq\emptyset$ and
	$k\leq\dim(M)\leq 2k-t$, so $1\leq a\leq k-t+1$.
	
	If $a\leq k-t$, then by Lemma \ref{2609221} (ii) and (iii),
	we have
	$$
	|\mathcal G|
	\leq q^{a(k-t)}{k-1\brack t-1}{a\brack1}
	{k-t+1\brack1}^{k-t-a}
	\leq q^{(k-t)(k-t+1)}{k-1\brack t-1},
	$$
	where the last inequality follows from
	${a\brack1}<q^a$ and
	${k-t+1\brack1}<q^{k-t+1}$.
	If $a=k-t+1$, the same upper bound follows directly from
	Lemma \ref{2609221} (ii) and $\Delta_k(\mathcal F)=1$.
	
	Set $\mathcal{H}=\mathcal{F}\setminus\mathcal{G}$. Since 
	every $F\in\mathcal F$ satisfies $\dim(F\cap T)\geq t-1$, we know
	$
	\mathcal{H}
	=\{F\in\mathcal{F}:\dim(F\cap T)\geq t\}.
	$
	Fix $G\in\mathcal G$. Then $G$ and $T$ are $t$-covers
	of $\mathcal{H}$, with $\dim(G\cap T)=t-1$.
	By Lemma \ref{2609221} (iii) and $a\leq k-t+1$, for any $0\leq s\leq k$, we have	
	$\Delta_s(\mathcal{H})
	\leq\Delta_s(\mathcal F)
	\leq{k-t+1\brack1}^{k-s}$,
	where the case $s=k$ follows from $\Delta_k(\mathcal F)=1$.
	Applying Lemma \ref{2608181}, we obtain
	$$
	|\mathcal{H}|
	\leq
	\sum_{i=0}^{t-1}
	q^{2i(i+1)}
	{t-1\brack i}
	{k-t+1\brack i+1}
	{k-t\brack i+1}
	{k-t+1\brack1}^{k-t-i-1}.
	$$
	Combining the bounds for $\mathcal{G}$ and $\mathcal{H}$
	and applying Lemma \ref{2609193} gives the desired bound.
\end{proof}

\subsection{$1\leq N\leq k-t-2$}

For $N=k-t-2\geq 1$, the final fingerprint consists of
$(t+2)$-subspaces, and Lemma \ref{2609091} gives the required bound
on its size.
Finally, when $1\le N\le k-t-3$, direct estimates suffice, although
the numerical comparisons are lengthy.

\begin{Lemma}
Suppose that $1\leq N=k-t-2$. Then 
	$
	\left|\mathcal{F}\right|<m(n,k,t).
	$
\end{Lemma}

\begin{proof}
	
	Recall that $\mathcal{S}_{k-t-2}=\mathcal{X}_{k-t-2}\subseteq {V\brack t+2}$. Since each member of $\mathcal{S}_{k-t-2}$ is a minimal $t$-cover of $\mathcal{S}_{k-t-2}$, Lemma \ref{2609091} gives
	$\left|\mathcal{S}_{k-t-2}\right|\leq {t+4\brack 2}$.
	
	From Lemma \ref{2609061} and $m_{k-t-2}=t+2$, we obtain
	$
	t+2\leq \tau=m_{0}\leq k-1$. For $W\in\mathcal S_{k-t-2}$, we have
	$|\mathcal F[W]|\le{k-t+1\brack1}^{\tau-t-2}
	{n-\tau\brack k-\tau}$
	by Corollary~\ref{key} when $\tau\ge t+3$,
	and by direct counting when $\tau=t+2$.
	Summing over $W$ gives
	\begin{equation*}
		\begin{aligned}
			\left|\mathcal{F}[\mathcal{X}_{k-t-2}]\right|
			\leq\sum_{W\in\mathcal{S}_{k-t-2}}\left|\mathcal{F}[W]\right|
			\leq {t+4\brack 2}
			{k-t+1\brack 1}^{\tau-t-2}
			{n-\tau\brack k-\tau}.
		\end{aligned}
	\end{equation*}
	
	Suppose $k=t+3$. Then $N=1$ and $\tau=t+2$. By Lemma \ref{2609056}, we have $
	\left|\mathcal{X}_{0}\right|\leq\varphi(k,t,t+2,t+2,0)$. Substituting the two bounds into \eqref{equdecomp} and applying
	Lemma~\ref{2609062}  gives
	$$
	\left|\mathcal{F}\right|
	\leq\varphi(t+3,t,t+2,t+2,0)+{t+4\brack 2}{n-t-2\brack 1}
	<m(n,t+3,t).
	$$
	
	Suppose $k\geq t+4$. We use Lemma~\ref{2609061} (i)--(iii) to bound the
	early layers $\bigcup_{i\leq k-t-3}\mathcal{F}[\mathcal{X}_i]$ by $\Phi_2(n,k,t,\tau,k-t-3)$. It follows that
	\begin{equation*}
		\left|\mathcal{F}\right|\leq\Phi_{2}(n,k,t,\tau,k-t-3)+{t+4\brack 2}
		{k-t+1\brack 1}^{\tau-t-2}
		{n-\tau\brack k-\tau}.
	\end{equation*}
	By Lemma \ref{2609216}, the desired result holds.
\end{proof}
\begin{Lemma}\label{2609231}
Suppose $1\leq N\leq k-t-3$. Then
	$
	\left|\mathcal{F}\right|<m(n,k,t).
	$
\end{Lemma}

\begin{proof}
	Since $\mathcal S_N=\mathcal X_N\subseteq{V\brack k-N}$,
	we have $m_N=k-N$, and Lemma~\ref{2609061} (i) gives
	$|\mathcal X_N|\le\psi(k,t,k-N,N)$.
	For the early layers, Lemma~\ref{2609061} (i) and (ii) give
	$|\mathcal X_0|\le\varphi(k,t,k-1,k,0)$ and
	$|\mathcal X_i|\le\psi(k,t,k-i-1,i)$ for $1\le i<N$.
	Substituting these bounds into \eqref{equdecomp}, we obtain
	$$
	|\mathcal F|\le\Phi_2(n,k,t,k-1,N-1)
	+\psi(k,t,k-N,N){n-k+N\brack N}.
	$$
	Lemma~\ref{260907} proves the result if
	$q\ge3$, or $n\ge2k+2$, or $N\ge2$.
	
	Suppose $q=2$, $n=2k+1$ and $N=1$. By Lemma~\ref{2609056}, we get 
	$|\mathcal X_0|\le\varphi(k,t,k-1,k-1,0)$. From  Lemma \ref{2608182} (i) and Lemma \ref{2609061} (i), we obtain
	$|\mathcal X_1|\le\varphi(k,t,k-1,k-1,1)$. 
	Substituting these bounds into \eqref{equdecomp} gives
	$$
	|\mathcal F|\le\varphi(k,t,k-1,k-1,0)
	+\varphi(k,t,k-1,k-1,1){k+2\brack1}.
	$$
	Lemma~\ref{2609099} completes the proof.
\end{proof}

\subsection{$N=k-t-1$}

Here we use the structural information about $\mathcal F$ encoded by the minimal $t$-covers and the final fingerprint to identify
the extremal families as $\mathcal A(Z)$.

\begin{Lemma}\label{2609232}
Suppose that $N=k-t-1$. Then there exists a $(t+2)$-subspace $Z$ of $V$ such that the
	following hold.
	\begin{itemize}
		\item[\normalfont(i)] $\mathcal{S}_{k-t-2}=\mathcal{S}_{k-t-1}$, and $\mathcal{S}_{k-t-1}\subseteq {Z\brack t+1}$.
		\item[\normalfont(ii)]  Either
		$
		\left|\mathcal{F}\right|
		<m(n,k,t)$ or $\mathcal{F}=\mathcal{A}(Z)$. 
	\end{itemize} 
\end{Lemma}

\begin{proof}

	(i) Note that $N=k-t-1\geq 1$. Lemma \ref{2609056} implies 
	$
	\mathcal{S}_{k-t-2}\setminus\mathcal{X}_{k-t-2}
	=\mathcal{S}_{k-t-1}.
	$
	
	Since each member of $\mathcal{S}_{k-t-1}$ is a minimal $t$-cover of $\mathcal{S}_{k-t-1}$, there exist $A,B\in\mathcal{S}_{k-t-1}$ such that $\dim(A\cap B)=t$. Since $A\cap B$ is a proper subspace of $A$, it is not a $t$-cover of $\mathcal{S}_{k-t-1}$. Choose $C\in\mathcal{S}_{k-t-1}$ with $\dim(A\cap B\cap C)<t$. Set $Z=A+B$. Then $\dim(Z)=t+2$, 
	$$
	t+1\geq\dim((A\cap C)+(B\cap C))
	\geq 2t-\dim(A\cap B\cap C)\geq t+1.
	$$
	Hence  $C\subseteq Z$ and $\dim(A\cap B\cap C)=t-1$.
	
	As $A,B,C\in\mathcal S_{k-t-2}$, every
	$W\in\mathcal S_{k-t-2}$ satisfies $\dim(W\cap Z)\ge t+1$.
	Otherwise, $
	W\cap A=W\cap B=W\cap C=W\cap Z$ is a $t$-subspace of $A\cap B\cap C$, a contradiction. It follows that every member of ${Z\brack t+1}$ is a $t$-cover of $\mathcal{S}_{k-t-2}$. Pick $T\in{W\cap Z\brack t+1}$. By the minimality of $W$, we know $W=T$. Therefore $\mathcal{S}_{k-t-2}\subseteq {Z\brack t+1}$, which implies $\mathcal{X}_{k-t-2}=\emptyset$ and $\mathcal{S}_{k-t-2}=\mathcal{S}_{k-t-1}$.
	
	(ii) 
	We divide the proof into three cases: $\tau\geq t+2$;  $\tau=t+1$ and $|\mathcal{T}_{t}(\mathcal{F})|\leq {2\brack 1}$;  $\tau=t+1$ and $|\mathcal{T}_{t}(\mathcal{F})|>{2\brack 1}$. We prove $
	\left|\mathcal{F}\right|
	<m(n,k,t)$  in the first two cases, 
	and $\mathcal{F}=\mathcal{A}(Z)$ in the last.
	
	Assume  $\tau\geq t+2$.  Lemma~\ref{2609061} (ii) gives $t+2\le\tau=m_0\le k-1$. So $k\ge t+3$. By (i), $\mathcal X_{k-t-2}=\emptyset$. Then by Lemma \ref{2609061} (i)--(iii), we bound the early layers $\bigcup_{i\leq k-t-3}\mathcal{F}[\mathcal{X}_i]$ by $\Phi_2(n,k,t,\tau,k-t-3)$. For the final layer, Corollary~\ref{key} gives
	$|\mathcal F[W]|\le{k-t+1\brack1}^{\tau-t-1}
	{n-\tau\brack k-\tau}$ for every $W\in{Z\brack t+1}$. By summing over these $W$ and using  Lemma \ref{2608182} (ii),
	we obtain
$$
	|\mathcal F|\le\Phi_2(n,k,t,\tau,k-t-3)
	+{t+2\brack1}{k-t+1\brack1}^{\tau-t-1}{n-\tau\brack k-\tau}.
$$
	From Lemma \ref{2609216} and Table \ref{2609176} (i), we obtain $
	\left|\mathcal{F}\right|
	<m(n,k,t)$.

	Assume $\tau=t+1$. Since $m_{0}=m_{k-t-1}=t+1$, Lemma \ref{2609061}  gives $m_{i}=t+1$ for $0\leq i\leq k-t-1$. Pick $T\in\mathcal{T}_{t}(\mathcal{F})$. Then $T\in\mathcal{S}_{0}$. Suppose $T\in\mathcal{S}_{i}$ for some $0\leq i\leq k-t-2$. Since $\dim(T)=t+1<k-i$, we have $T\notin\mathcal{X}_{i}$. By the definition of a fingerprint, there exists $X\in\mathcal{S}_{i+1}$ with $X\subseteq T$. From $m_{i+1}=t+1=\dim(T)$, we obtain $X=T$. Applying this argument repeatedly yields
	$\mathcal{T}_{t}(\mathcal{F})
	\subseteq\mathcal{S}_{k-t-1}
	\subseteq{Z\brack t+1}$.
	
	Suppose $|\mathcal{T}_{t}(\mathcal{F})|\leq {2\brack 1}$. For $k\ge t+3$, Lemma~\ref{2609061} (i) and $m_i=t+1$ give
\begin{equation}\label{2609057}
	\sum_{i=0}^{k-t-3}|\mathcal F[\mathcal X_i]|
	\le\sum_{i=0}^{k-t-3}|\mathcal X_i|{n-k+i\brack i}
	\le\Phi_1(n,k,t,t+1,t+1,k-t-3).
\end{equation}
For $k=t+2$, the sums are empty.
	
	For $H\in{Z\brack t+1}\setminus\mathcal{T}_{t}(\mathcal{F})$, there exists $G\in\mathcal{F}$ such that $\dim(H\cap G)<t$. Since $Z$ contains a member of $\mathcal{T}_{t}(\mathcal{F})$, we have $\dim(H\cap G)=t-1$. We further conclude $$\mathcal{F}[H]\subseteq \bigcup_{U\in{H+G\brack t+2},\ H\subseteq U}\mathcal{F}[U].$$ This implies  $\left|\mathcal{F}[H]\right|\leq {k-t+1\brack 1}{n-t-2\brack k-t-2}$. 
	For $W\in\mathcal{T}_{t}(\mathcal{F})$, we have
	$|\mathcal{F}[W]|\leq {n-t-1\brack k-t-1}$. Hence 
	\begin{equation}\label{2609058}
		\left|\mathcal{F}[\mathcal{S}_{k-t-2}]   \right|\leq\left({t+2\brack 1}-|\mathcal{T}_{t}(\mathcal{F})|\right){k-t+1\brack 1}{n-t-2\brack k-t-2}+|\mathcal{T}_{t}(\mathcal{F})|{n-t-1\brack k-t-1}.
	\end{equation}
	Note that
	$$
	\frac{{k-t+1\brack 1}{n-t-2\brack k-t-2}}{{n-t-1\brack k-t-1}}=\frac{(q^{k-t+1}-1)(q^{k-t-1}-1)}{(q-1)(q^{n-t-1}-1)}\leq \frac{1}{q^{n-2k+t-1}}\leq 1.
	$$
Then the RHS of \eqref{2609058} is increasing in
$|\mathcal T_t(\mathcal F)|$, and takes the maximum value when $|\mathcal T_t(\mathcal F)|={2\brack1}$. Combining \eqref{2609057}, \eqref{2609058} and Lemma \ref{2608182} (ii)
gives, for $k\ge t+3$,
	$$
	\begin{aligned}
	|\mathcal{F}|&\leq\Phi_{1}(n,k,t,t+1,t+1,k-t-3)
	+{2\brack 1}{n-t-1\brack k-t-1}
	+q^{2}{t\brack 1}{k-t+1\brack 1}{n-t-2\brack k-t-2},
	\end{aligned}
	$$
	and for $k=t+2$,  $|\mathcal{F}|\leq {2\brack 1}{n-t-1\brack 1}
	+q^{2}{t\brack 1}{3\brack 1}$.
	It follows from Lemma \ref{2609052} that $
	\left|\mathcal{F}\right|
	<m(n,k,t)$, as desired.

	Suppose $\left|\mathcal{T}_{t}(\mathcal{F})\right|>{2\brack 1}$. Recall that $\mathcal{T}_{t}(\mathcal{F})
	\subseteq{Z\brack t+1}$. If $\dim(F\cap Z)\leq t$ for some $F\in\mathcal{F}$, then by $|\mathcal{T}_{t}(\mathcal{F})|>{2\brack 1}$,  there exists $T\in\mathcal{T}_{t}(\mathcal{F}) $ such that $\dim(F\cap T)<t$, a contradiction. Hence $\mathcal{F}\subseteq \mathcal{A}(Z)$. It follows from the maximality of $\mathcal{F}$ that $\mathcal{F}= \mathcal{A}(Z)$.
\end{proof}

\subsection{$N=k-t$}

This case requires a more detailed analysis,
since extremal and nearly extremal configurations can occur. We begin with an auxiliary counting lemma.

\begin{Lemma}\label{2609252}
	Let $\mathcal S\subseteq{V\brack\leq t+2}$ be non-empty.
	Suppose that every member of $\mathcal S$ is a minimal
	$t$-cover of $\mathcal S$.
	If $A\in\mathcal S$ has dimension $t+1$ and $X\in{A\brack t}$, then
	$$
	|\{W\in\mathcal S:\dim(W)=t+2,\ X\nsubseteq W\}|
	\leq q^3{t\brack1}{2\brack1}^2.
 $$
\end{Lemma}
\begin{proof}
	Since $\mathcal S$ is an antichain, we get
	$\dim(S\cap A)=t$ for any $S\in\mathcal S\setminus\{A\}$.
	Write $\mathcal W=\mathcal S\cap{V\brack t+2}$.   
	  Then $\{W\in\mathcal{W}:X\nsubseteq W\}\subseteq \cup_{E}\mathcal{W}[E]$, where $E$ ranges over all $t$-subspaces of $A$ satisfying $E\neq X$. It is sufficient to show $\left|\mathcal{W}[E]\right|\leq q^{2}{2\brack 1}^{2}$ for any $E\in{A\brack t}$.

	Fix $E\in{A\brack t}$.  We may assume $\mathcal{W}[E]$ is non-empty. 
	Suppose  $E\subseteq Z\subsetneq W\in\mathcal W[E]$.
	By the minimality of $W$, some $B_{Z}\in\mathcal{S}\setminus \{A\}$ satisfies
	$\dim( B_{Z}\cap Z)<t$. 
	 Moreover, by $\dim(B_{Z}\cap A)=t$, we know $\dim(B_{Z}\cap Z)=\dim(B_{Z}\cap E)=t-1$.  We further conclude $A=E+(A\cap B_{Z})\subseteq Z+B_{Z}$, which implies $Z+A\subseteq Z+B_{Z}$.

   We first set $Z=E$. Then $A\subseteq E+B_{E}$. Recall that $A\nsubseteq W$ for any $W\in\mathcal{W}$.  Then $\mathcal{W}[E]\subseteq \cup_{H} \mathcal{W}[H]$, where $H$ ranges over all $(t+1)$-subspaces of $E+B_{E}$ satisfying $E\subseteq H$ and $A\neq H$. Pick such an $H$ with $|\mathcal{W}[H]|>0$. Then set $Z=H$ in the above process.  Notice that  $\dim(H+A)=t+2$ and $H+A\subseteq H+B_{H}$.  It follows that $\mathcal{W}[H]\subseteq \cup_{U} \mathcal{W}[U]$, where $U$ ranges over all $(t+2)$-subspaces of $H+B_{H}$ satisfying $H\subseteq U$ and $A+H\neq U$. Therefore, we have $\left|\mathcal{W}[E]\right|\leq q^{2}{2\brack 1}^{2}$, as desired.
   \end{proof}

\begin{Lemma}\label{2609121}
	Suppose that $N=k-t$. 
	Then $\mathcal S_{k-t}=\{X\}$ for some $X\in{V\brack t}$, and the following hold.
	\begin{itemize}
		\item[\normalfont(i)] Set $j=\min\{i:\mathcal S_i=\{X\}\}$.
		Then $1\leq j\leq k-t-1$.
		\item[\normalfont(ii)] If $\tau\geq t+2$, then
		$|\mathcal F|<m(n,k,t)$.
		\item[\normalfont(iii)] Suppose that $\tau=t+1$.
		Set $M=\sum_{T\in\mathcal T_t(\mathcal F)}T$.
		If $\dim(M)\leq k$, then
		$|\mathcal F|<m(n,k,t).$
		If $\dim(M)=k+1$, then $\mathcal F=\mathcal H(X,M)$.
	\end{itemize}
\end{Lemma}
\begin{proof}
	The family $\mathcal S_{k-t}$ is a $t$-intersecting antichain
	consisting of $t$-subspaces, so $\mathcal S_{k-t}=\{X\}$ for some $X\in{V\brack t}$.
	
	(i) By property (ii) in the definition of a fingerprint,
	each member of the non-empty family
	$\mathcal S_{k-t-1}\setminus\mathcal X_{k-t-1}$ contains $X$.
	Its members have dimension at most $t$, so
	$\mathcal S_{k-t-1}\setminus\mathcal X_{k-t-1}=\{X\}$. 
	The $t$-intersection and antichain properties imply that
	$\mathcal S_{k-t-1}=\{X\}$, so $j\leq k-t-1$.
	Since $\mathcal S_0=\mathcal M_t(\mathcal F)$ and
	$\tau\geq t+1$, we also have $j\geq1$.
	
	For later use, note that every member of
	$\mathcal S_{j-1}\setminus\mathcal X_{j-1}$ contains $X$,
	whereas $X\notin\mathcal S_{j-1}$ by the choice of $j$.
	If every member of $\mathcal X_{j-1}$ also contained $X$,
	then $X$ would be a $t$-cover of $\mathcal S_{j-1}$,
	contrary to the minimality of its members.
	Hence $\mathcal X_{j-1}\setminus\mathcal X_{j-1}[X]\neq\emptyset$. For $j\le i<N$, we have $\mathcal S_i=\{X\}$ and
	$\mathcal X_i=\emptyset$, so the early-layer union in
	\eqref{equdecomp} ends at $j-1$. Hence
	\begin{equation}\label{2609251}
		\mathcal F=\left(\bigcup_{i=0}^{j-1}\mathcal F[\mathcal X_i]\right)\cup\mathcal F[X].
	\end{equation}
	
	(ii) By Lemma~\ref{2609061} (ii), we know
	$t+2\leq\tau=m_{0}\leq k-1$, and hence $k\geq t+3$.	
	For the early layers,	Lemma~\ref{2609061} (i)--(iii) implies
	\begin{equation}\label{equN=k-t}
		|\mathcal X_0|\leq\varphi(k,t,\tau,k,0), \ 	|\mathcal F[\mathcal X_i]|
		\leq\psi(k,t,k-i-1,i){n-k+i\brack i},\qquad 1\leq i<j.
	\end{equation}

	Suppose that $j\le k-t-2$.
	We apply Corollary~\ref{key} to $\mathcal F$ to bound $|\mathcal F[X]|$.
	Combining this bound with \eqref{2609251} and \eqref{equN=k-t}  gives
	$$
	|\mathcal F|
	\leq\Phi_2(n,k,t,\tau,k-t-3)
	+{k-t+1\brack1}^{\tau-t}
	{n-\tau\brack k-\tau}.
	$$
	If $(n,k,t,q)\neq(9,4,1,2)$, then Lemma \ref{2609216} gives the desired result. If $(n,k,t,q)=(9,4,1,2)$, then $\tau=3$ and $j=1$.  Lemma~\ref{2609061} (i) and Corollary~\ref{key} yield
	$$
	|\mathcal F|\leq\varphi(4,1,3,4,0)
	+{4\brack1}\left({7\brack2}-2^8{3\brack2}\right)
	=35400<35731=h(9,4,1).
	$$
	
	It remains to consider $j=k-t-1$.
	Every member of $\mathcal S_{k-t-2}\setminus\mathcal X_{k-t-2}$
	has dimension $t+1$ and contains $X$.
	Choose $W\in\mathcal X_{k-t-2}\setminus\mathcal X_{k-t-2}[X]$.
	Since $W$ $t$-intersects a member of
	$\mathcal S_{k-t-2}\setminus\mathcal X_{k-t-2}$ but does not contain $X$,
	we have $\dim(W\cap X)=t-1$ and $\dim(X+W)=t+3$. We apply Lemma~\ref{2608182} (ii) with $i=k-t-2$
	and estimate the remaining family $\mathcal F[\mathcal S_{k-t-2}]$
	according to whether its members contain $X$.
	Every member of $\mathcal F[\mathcal S_{k-t-2}]$ $t$-intersects $W$.
	If it contains $X$, then it contains a $(t+1)$-subspace $H$
	with $X\subseteq H\subseteq X+W$.
	There are ${3\brack1}$ choices for $H$.
	Applying Corollary~\ref{key} and summing over these choices gives
	$$
	|\mathcal F[\mathcal S_{k-t-2}]\cap\mathcal F[X]|
	\leq{3\brack1}{k-t+1\brack1}^{\tau-t-1}
	{n-\tau\brack k-\tau}.
	$$
	On the other hand, each member of
	$\mathcal F[\mathcal S_{k-t-2}]\setminus\mathcal F[X]$
	contains a member of $\mathcal X_{k-t-2}\setminus\mathcal X_{k-t-2}[X]$. There are at most
	$q^3{t\brack1}{2\brack1}^2$ such $(t+2)$-subspaces
by applying  Lemma~\ref{2609252} with $\mathcal{S}=\mathcal{S}_{k-t-2}$.  
	Consequently, \eqref{2609251} and \eqref{equN=k-t} imply
	\begin{equation*}
		|\mathcal F|
		\leq\Phi_2(n,k,t,\tau,k-t-3)
		+{3\brack1}{k-t+1\brack1}^{\tau-t-1}
		{n-\tau\brack k-\tau}
		+q^3{t\brack1}{2\brack1}^2{n-t-2\brack k-t-2}.
	\end{equation*}
	Lemma \ref{2609216} and Table \ref{2609176} (ii) now give $|\mathcal F|<m(n,k,t)$.
	
	(iii) Since $\tau=t+1$, we have
	$\mathcal T_t(\mathcal F)=\mathcal S_0\cap{V\brack t+1}$.
	For $i<j$, every member of $\mathcal X_i$ has dimension at least $t+2$.
	Moreover, if some $\mathcal S_i$ contained a $t$-subspace,
	then its $t$-intersection and antichain properties would make it a singleton.
	That subspace would persist to $\mathcal S_{k-t}=\{X\}$,
	so $\mathcal S_i=\{X\}$.
	Thus no member of $\mathcal T_t(\mathcal F)$ can be replaced by a proper
	subspace before index $j$, and
	$$
	\mathcal T_t(\mathcal F)\subseteq\mathcal S_{j-1}\setminus\mathcal X_{j-1}.
	$$
	In particular, every member of $\mathcal T_t(\mathcal F)$ contains $X$,
	so $\dim(F\cap X)\geq t-1$ for every $F\in\mathcal F$.
	
	Set
	$
	\mathcal F_1=\mathcal F\setminus\mathcal F[X]$ and $
	\mathcal F_2=\mathcal F[X]\setminus\mathcal F[\mathcal T_t(\mathcal F)].
	$
	Then
	$\mathcal F=\mathcal F_1\cup\mathcal F[\mathcal T_t(\mathcal F)]\cup\mathcal F_2$,
	where the three subfamilies are pairwise disjoint.

	The nontriviality of $\mathcal F$ ensures that
	$\mathcal F_1\neq\emptyset$.
	For every $D\in\mathcal F_1$ and
	$T\in\mathcal T_t(\mathcal F)$, we have
	$\dim(D\cap X)=t-1$ and $T=X+(D\cap T)\subseteq X+D$.
	Hence $M\subseteq X+D$ and $t+1\leq\dim(M)\leq k+1$.
	Conversely, every $(t+1)$-subspace $T$ with $X\subseteq T\subseteq M$
	is a $t$-cover of $\mathcal F$:
	it $t$-intersects every member of $\mathcal F[X]$, and for 
	$D\in\mathcal F_1$, the inclusion
	$T+D\subseteq X+D$ gives $\dim(T\cap D)\geq t$.
	Therefore, we have
	\begin{equation}\label{2609255}
		\mathcal T_t(\mathcal F)=\left\{T\in{M\brack t+1}:X\subseteq T\right\}.
	\end{equation}
	
	Suppose first that $\dim(M)=k+1$.
	Then $M=X+D$ for every $D\in\mathcal F_1$,
	so $\mathcal F_1\subseteq{M\brack k}$.
	Fix such a $D$. For every $F\in\mathcal F[X]$, we have
	$\dim(F\cap D)\geq t$ and $\dim(X\cap D)=t-1$,
	which imply $\dim(F\cap M)\geq t+1$.
	Thus $\mathcal F\subseteq\mathcal H(X,M)$,
	and equality follows from the maximality.
	
	In what follows, suppose that $\dim(M)\leq k$.
	By (\ref{2609255}), we get
	\begin{equation}\label{2609253}
			|\mathcal F[\mathcal T_t(\mathcal F)]|
		\leq{\dim(M)-t\brack1}{n-t-1\brack k-t-1}.
	\end{equation}
	Fix $D\in\mathcal F_1$.
	Every member of $\mathcal F_2$
	contains a $(t+1)$-subspace $H$ with
	$X\subseteq H\subseteq X+D$ and $H\nsubseteq M$.
	There are ${k-t+1\brack1}-{\dim(M)-t\brack1}$ choices for $H$.
	Since $H$ is not a $t$-cover of $\mathcal F$, there exists
	$D_H\in\mathcal F$ with $\dim(H\cap D_H)=t-1$.
	Counting the $(t+2)$-subspaces between $H$ and $H+D_H$ gives
	$|\mathcal F[H]|\leq{k-t+1\brack1}{n-t-2\brack k-t-2}$.
	Consequently,
	\begin{equation}\label{2609256}
		 |\mathcal F_2|
			\leq\left({k-t+1\brack1}-{\dim(M)-t\brack1}\right)
			{k-t+1\brack1}{n-t-2\brack k-t-2}.
	\end{equation}
	
	If $\dim(M)=t+1$, then $\mathcal T_t(\mathcal F)=\{M\}$.
	For $i<j$, the subspace $M$ belongs to $\mathcal S_i$, so every member of
	$\mathcal X_i\setminus\mathcal X_i[X]$ contains a $t$-subspace of $M$
	different from $X$.
	By Lemma~\ref{2608182} (iii),
	$$
	|\mathcal X_i\setminus\mathcal X_i[X]|
	\leq\left({t+1\brack t}-1\right){k-t-i+1\brack1}^{k-t-i}
	=q{t\brack1}{k-t-i+1\brack1}^{k-t-i}.
	$$
	Using (\ref{2609251}),  we obtain $\mathcal{F}_{1}\subseteq \cup_{i=0}^{j-1}\mathcal{F}[\mathcal{X}_{i}\setminus \mathcal{X}_{i}[X]]$. It follows from $j\leq k-t-1$ that
	\begin{equation*}
		|\mathcal F_1|
		\leq q{t\brack1}\Phi_1(n,k,t,t,t,k-t-2).
	\end{equation*}
	Together with (\ref{2609253}), (\ref{2609256}) and Lemma~\ref{2609258},
	this proves $|\mathcal F|<m(n,k,t)$.
	
	Finally, suppose that $t+2\leq\dim(M)\leq k$.
	For every $F\in\mathcal F_1$,
	the inclusion $M\subseteq X+F$ gives $\dim(F\cap M)=\dim(M)-1$.
	There are $q^{\dim(M)-t}{t\brack1}$ hyperplanes of $M$ not containing $X$.
	For each such hyperplane $L$, the number of $k$-subspaces $F$ with
	$F\cap M=L$ is $q^{k-\dim(M)+1}{n-\dim(M)\brack k-\dim(M)+1}$.
	Hence, if $t+3\leq\dim(M)\leq k$, then
	\begin{equation*}
		|\mathcal F_1|
		\leq q^{k-t+1}{t\brack1}{n-\dim(M)\brack k-\dim(M)+1}.
	\end{equation*}
	If $\dim(M)=t+2$, then each such $L$ has dimension $t+1$
	and is not a $t$-cover of $\mathcal F$ by (\ref{2609255}).
	Every member of $\mathcal F$ meets $M$ in dimension at least $t$,
	and hence meets $L$ in dimension at least $t-1$.
	The preceding argument for $H$ gives
	$|\mathcal F[L]|\leq{k-t+1\brack1}{n-t-2\brack k-t-2}$, and thus
	$$
	|\mathcal F_1|
	\leq q^2{t\brack1}{k-t+1\brack1}{n-t-2\brack k-t-2}.
	$$
	Combining these estimates with  (\ref{2609253}), (\ref{2609256}) and using
	Lemma~\ref{2609257} completes the proof.
\end{proof}

  \begin{proof}[\textnormal{\textbf{Proof of Theorem \ref{main}}}]
	
Let  $\mathcal F\subseteq{V\brack k}$ be a nontrivial
$t$-intersecting family of maximum size. Then certainly $\mathcal F$ is maximal.
 By Lemmas \ref{2609122}--\ref{2609232}, \ref{2609121}, we know $\mathcal{F}=\mathcal{A}(Z)$ for some $Z\in{V\brack t+2}$, or  $\mathcal{F}=\mathcal{H}(X, M)$ for some $M\in{V\brack k+1}$ and $X\in{M\brack t}$. This together with Lemma \ref{2609131} finishes the proof.
\end{proof}
  
  \noindent{\bf Acknowledgment.}	
  B. Lv is supported by the National Natural Science Foundation of China (12571347, 12131011) and the Beijing Natural Science Foundation (1252010). K. Wang \\is supported by the National Natural Science Foundation of China (12571347, 12131011) and the Beijing Natural Science Foundation (1252010, 1262010). \medskip

  \noindent{\bf Declaration on the use of AI.}\; The authors used AI tools to assist with improving the language and deriving some   technical estimates. All AI-assisted revisions were reviewed and approved by the authors before being
  incorporated into the manuscript.

\appendix

\section{Technical estimates}\label{s4}

This appendix is devoted to proving some inequalities concerning  Gaussian binomial 
coefficients. The first one  is   the well-known $q$-Vandermonde identity.

\begin{Lemma}\label{2609151}
	Let $a$, $b$ and $c$  be non-negative integers with $a+b\geq c\geq 0$. Then
	$${a+b\brack c}=\sum_{i=0}^{c}q^{i(b-c+i)}{a\brack i}{b\brack c-i}.$$
\end{Lemma}

From now on, lowercase letters   denote positive integers, unless otherwise stated.  The following lemma can be easily  verified.
 
\begin{Lemma}\label{2604282}
	Let  $n>k$. Then
	$$q^{n-k}<\frac{q^{n}-1}{q^{k}-1}< \frac{q^{n-k+1}}{q-1}, \ \ q^{k(n-k)}<{n\brack k}<\frac{q^{k(n-k+1)}}{(q-1)^{k}}.$$
\end{Lemma}

\begin{Lemma}\label{2608201}
   The following holds.
   $${t+4\brack 2}>\max\left\{q^{6}{t+1\brack 2}+q^{4}{3\brack 1}{t-1\brack 1}+{3\brack 1}{2\brack 1}^{2},\ {2\brack 1}{3\brack 1}^{2}+q^{4}{t-1\brack 1}\right\}. $$
\end{Lemma}
\begin{proof}
	Substituting $(a,b,c)=(t+1,3,2)$ into Lemma \ref{2609151}, we obtain
	$${t+4\brack 2}=q^{6}{t+1\brack 2}+{3\brack 1}\left(q^{2}{t+1\brack 1}+1\right)>q^{6}{t+1\brack 2}+{3\brack 1}\left(q^{4}{t-1\brack 1}+{2\brack 1}^{2}\right).$$ 
	
	 Applying Lemma \ref{2609151} with $(a,b,c)=(t-1,5,2)$, we have 
	$${t+4\brack 2}={5\brack 2}+q^{4}{t-1\brack 1}{5\brack 1}+q^{10}{t-1\brack 2}\geq {5\brack 2}+q^{4}{t-1\brack 1}>{2\brack 1}{3\brack 1}^{2}+q^{4}{t-1\brack 1}.$$
	This completes the proof.
\end{proof}

Recall that $\psi(k,t,a,i)$ and $\varphi(k,t,a,b,i)$ are defined in (\ref{2609111}) and (\ref{2609112}), respectively;   $\Phi_{1}(n,k,t,a,b,j)$ and $\Phi_{2}(n,k,t,a,j)$ are defined in (\ref{2609115}) and (\ref{2609113}), respectively.

We  adopt the convention that  an empty sum and an empty product are equal to $0$ and $1$, respectively. 
For convenience,  write $$r=k-t,\ \ x=\max\{t+2, r+1\}. $$

\begin{Lemma}\label{2609152}
	Let $k\geq t+3$ and $n\geq 2k+1$. Then 
	$$\frac{{r+1\brack 1}^{r-1}}{{n-t-1\brack r-1}}\leq \frac{q^{r+1}-1}{q^{r+1}}\cdot \frac{1}{(q-1)^{r-1}q^{(r-1)(n-2k+t-1)}}\cdot\prod_{i=1}^{r-1}\frac{q^{i}-1}{q^{i}}.$$
\end{Lemma}

\begin{proof}
	By routine computations, we know that ${r+1\brack 1}^{r-1}/{n-t-1\brack r-1}$ is equal to
	\begin{equation*}
		\begin{aligned}
			\frac{1}{(q-1)^{r-1}q^{(r-1)(n-2k+t-1)}}\cdot\left(
			\frac{q^{r+1}-1}{q^{r+1}}\right)^{r-1}\cdot \prod_{i=1}^{r-1}\frac{q^{i}-1}{q^{i}}\prod_{i=1}^{r-1}\frac{q^{n-k+i}}{q^{n-k+i}-1}.
		\end{aligned}
	\end{equation*}
	Applying the Weierstrass product inequality, we get 
	$$\prod_{i=1}^{r-1}\left(1-\frac{1}{q^{n-t-i}}\right)\geq 1-\sum_{i=1}^{r-1}\frac{1}{q^{n-t-i}}\geq 1-\frac{1}{q^{n-k}(q-1)} \geq 1-\frac{1}{q^{r+1}}.$$
	It follows from $r\geq 3$ that the desired result holds.
\end{proof}

\begin{Lemma}\label{2609153}
	Let $y\geq 1$, $k\geq t+3$ and $n\geq 2k+1$. Then 
	$$\frac{\varphi(k,t,t+1,k,0)}{{y\brack 1}{n-t-1\brack r-1}}\leq \frac{q^{t+1}-1}{q^{y}-1}\cdot\frac{(q^{2}-1)(q^{3}-1)}{q^{4}(q-1)^{r-1}}\cdot\frac{1}{q^{(r-1)(n-2k+t-2)}}.$$
	Moreover,  for $t+2\leq a\leq k-1$, we have 
	$$\frac{\varphi(k,t,a,k,0)}{{y\brack 1}{n-t-1\brack r-1}}\leq \frac{q^{r+1}-1}{q^{y}-1}\cdot \frac{(q^{2}-1)(q^{3}-1)(q^{4}-1)}{q^{9}(q-1)^{r-1}}\cdot\frac{1}{q^{(r-1)(n-2k-1)+t(k-a-1)}}, $$
	where   $ (q^{2}-1)(q^{3}-1)(q^{4}-1)/\bigl(q^{9}(q-1)^{r-1}\bigr)$ is decreasing in $q$, and it is at most $315/512$.
\end{Lemma}

\begin{proof}
  Recall that $\varphi(k,t,t+1,k,0)={r+1\brack 1}^{r-1}\left(q{t+1\brack 1}{r\brack 1}+1\right)$. Since $q{t+1\brack 1}{r\brack 1}+1$ is at most $q{t+1\brack 1}{r\brack 1}\left(1+1/ (q^{r+1}-1)\right)$, we have  
	\begin{equation*}
		\frac{\varphi(k,t,t+1,k,0)}{{y\brack 1}{n-t-1\brack r-1}}\leq \frac{q^{t+1}-1}{q^{y}-1}\cdot\frac{q^{r+1}}{(q-1)^{r}q^{(r-1)(n-2k+t-1)}}\cdot \prod_{i=1}^{r}\frac{q^{i}-1}{q^{i}}. 
	\end{equation*}
	It follows from $r\geq 3$ that the first inequality holds.
	
	 During the following proof, for $0\leq j\leq t$ and $1\leq i\leq t$,  write 
	$$A_{j}=q^{\binom{j}{2}}{t\brack j}{a-t\brack j}\prod_{i=1}^{j}(q^{i}-1),\ \  B_{i}=\frac{q^{3i-1}(q-1)(q^{r-i+1}-1)}{(q^{r+1}-1)(q^{i}-1)^{2}}. $$
	In this notation,  $\varphi(k,t,a,k,0)/ {r+1\brack 1}^{r}=\sum_{j=0}^{t}A_{j}\prod_{i=1}^{j}B_{i}$. 
	
	If $2\leq i\leq \min\{t, a-t\}$, then by $(q^{i}-1)^{2}/q^{2i}\geq (q^{2i}-2q^{i})/q^{2i}\geq 1-1/q$, we have 
	$$B_{i}=\frac{q^{2i-1}(q-1)}{(q^{i}-1)^{2}}\cdot\frac{q^{r+1}-q^{i}}{q^{r+1}-1}\leq \frac{q^{r+1}-q^{i}}{q^{r+1}-1}\leq 1-\frac{q^{2}-1}{q^{r+1}-1}\leq 1-\frac{1}{(q-1)q^{r}}\leq 1. $$
	It is well known that the number of $t\times (a-t)$ matrices over
	$\mathbb F_{q}$ of rank $j$ is $A_{j}$. This implies 
	$$\sum_{j=2}^{t}A_{j}=q^{(a-t)t}-A_{0}-A_{1}=(q^{a-t}-1)\sum_{i=1}^{t-1}(q^{(a-t)i}-q^{i})\geq \frac{q-1}{q}(q^{a-t}-1)\sum_{i=1}^{t-1}q^{(a-t)i},$$
	and the RHS of this inequality  equals $\frac{q-1}{q}\left(q^{t(a-t)}-q^{a-t}\right)$. 
	Combining these bounds gives
	\begin{equation*}
		\begin{aligned}
			&\ \frac{\varphi(k,t,a,k,0)}{{r+1\brack 1}^{r}}=  1+A_{1}B_{1}+\sum_{j=2}^{t}A_{j}\prod_{i=1}^{j}B_{i}\\
			\leq&\  1+A_{1}B_{1}+B_{1}\left(1-\frac{1}{(q-1)q^{r}}\right)\sum_{j=2}^{t}A_{j}
			=1+B_{1}\left(q^{t(a-t)}-1\right)-\frac{B_{1}}{(q-1)q^{r}}\sum_{j=2}^{t}A_{j}\\
			\leq&\  1+B_{1}\left(q^{t(a-t)}-1\right)-\frac{B_{1}}{q^{r+1}}\left(q^{t(a-t)}-q^{a-t}\right)=B_{1}\frac{(q^{r+1}-1)q^{t(a-t)}}{q^{r+1}}+1-B_{1}\left(1-\frac{q^{a-t}}{q^{r+1}}\right).
		\end{aligned}
	\end{equation*}
	Since $a-t\leq r-1$, we have $B_{1}\left(1-\frac{q^{a-t}}{q^{r+1}}\right)\geq B_{1}(1-\frac{1}{q^{2}})=(q+1)(q^{r}-1)/ (q^{r+1}-1)>1$. Thus $\varphi(k,t,a,k,0)/ {r+1\brack 1}^{r}\leq  B_{1}(q^{r+1}-1)q^{t(a-t)}/ q^{r+1}$. Note that 
	$$\frac{\varphi(k,t,a,k,0)}{{y\brack 1}{n-t-1\brack r-1}}=\frac{{r+1\brack 1}}{{y\brack 1}}\cdot \frac{\varphi(k,t,a,k,0)}{{r+1\brack 1}^{r}}\cdot\frac{{r+1\brack 1}^{r-1}}{{n-t-1\brack r-1}}.$$It follows from Lemma \ref{2609152} and $r\geq3$ that the second inequality holds. 
	
	In particular,  $(q^{2}-1)(q^{3}-1)(q^{4}-1)/(q^{9}(q-1)^{2})$  is   equal  to $$ \frac{q-1}{q^{2}}\left(1+\frac{1}{q}\right)^{2}\left(1+\frac{1}{q^{2}}\right)\left(\frac{1}{q}+\frac{1}{q^{2}}+\frac{1}{q^{3}}\right).$$ 
	By $r\geq 3$,  the claimed monotonicity in $q$ also follows.
\end{proof}

Let $k\geq t+2$, $n\geq 2k+1$, $1\leq i\leq r-1$.  By ${k-i-1\brack t}={k-i-1\brack r-i-1}$,  ${r-1\brack r-i-1}={r-1\brack i}$, we get
\begin{equation*}
	\frac{{k-i-1\brack t}{n-k+i\brack i}}{{n-t-1\brack r-1}}=\prod_{j=1}^{i}\frac{q^{r-j}-1}{q^{i-j+1}-1}\prod_{j=1}^{r-i-1}\frac{q^{k-i-j}-1}{q^{n-t-j}-1}.
\end{equation*}
Then $\psi(k,t,k-i-1, i){n-k+i\brack i}/\left({x\brack 1}{n-t-1\brack r-1}\right)$ is equal to
\begin{equation}\label{2609161}
	\frac{(q^{r-i+1}-1)^{r-i}}{(q-1)^{r-i-1}(q^{x}-1)}\prod_{j=1}^{i}\frac{q^{r-j}-1}{q^{i-j+1}-1}\prod_{j=1}^{r-i-1}\frac{q^{k-i-j}-1}{q^{n-t-j}-1}.
\end{equation}

\begin{Lemma}\label{2608317}
	Let  $k\geq t+4$ and $n\geq 2k+1$. Then 
	$$\frac{\sum_{i=1}^{r-3}\psi(k,t,k-i-1, i){n-k+i\brack i}}{{x\brack 1}{n-t-1\brack r-1}}\leq \frac{1}{(q-1)^{r}}\cdot\frac{1}{q^{2(n-2k-1)}}\cdot \frac{1}{q^{r-3}} \left( 1-\frac{1}{q^{r-3}}\right)\cdot\frac{q^{r}-4}{q^{r}}\leq \frac{3}{16}.$$
	Moreover, we have $\sum_{i=1}^{r-3}\psi(k,t,t+1, i){n-k+i\brack i}\leq 9{x\brack 1}{n-t-1\brack r-1}/112$.
\end{Lemma}
\begin{proof}
	Let $1\leq i\leq r-3$.	 
	By routine computations, we have  $$\prod_{j=1}^{i}\frac{q^{r-j}-1}{q^{i-j+1}-1}\leq \prod_{j=1}^{i}\frac{q^{r-j}-1}{(q-1)q^{i-j}}=\frac{q^{i(r-i)}}{(q-1)^{i}}\prod_{j=1}^{i}\frac{q^{r-j}-1}{q^{r-j}}\leq \frac{q^{i(r-i)}}{(q-1)^{i}}\cdot\frac{q^{r-1}-1}{q^{r-1}}.$$	Notice that $\left(1-\frac{2}{q^{r}}\right)\left(1-\frac{1}{q^{r}}\right)^{3}=\left(1-\frac{4}{q^{r}}\right)\left(1-\frac{1}{2q^{r}}\right)-\frac{q^{r}-14}{2q^{2r}}-\frac{7q^{r}-2}{q^{4r}}$. From $q^{r}\geq 16$, we get
	$$\left(1-\frac{1}{q^{r-1}}\right)\left(1-\frac{1}{q^{r-i+1}}\right)^{r-i}\leq \left(1-\frac{2}{q^{r}}\right)\left(1-\frac{1}{q^{r}}\right)^{3}\leq \left(1-\frac{4}{q^{r}}\right)\left(1-\frac{1}{q^{r+1}}\right).$$	It follows from (\ref{2609161}) and Lemma \ref{2604282} that 
	\begin{equation}\label{2609163}
		\begin{aligned}
			\frac{\psi(k,t,k-i-1, i){n-k+i\brack i}}{{x\brack 1}{n-t-1\brack r-1}}&\leq \frac{q^{r+1}-1}{q^{x}-1}\cdot\frac{q^{(r-i-1)(r+1)}}{(q-1)^{r-1}}\cdot \frac{q^{r}-4}{q^{r}}\prod_{j=1}^{r-i-1}\frac{q^{k-i-j}-1}{q^{n-t-j}-1}\\
			&\leq \frac{1}{(q-1)^{r-1}}\cdot\frac{q^{r}-4}{q^{r}}\cdot \frac{1}{q^{(r-i-1)(i+n-2k-1)}}.
		\end{aligned}
	\end{equation}
	For each  $1\leq i\leq r-3$, we have $i(r-i-1)-(r+i-3)=(i-1)(r-i-3)\geq 0$. Hence
	$$ \sum_{i=1}^{r-3}\frac{1}{q^{(r-i-1)(i+n-2k-1)}}\leq\frac{1}{q^{2(n-2k-1)}}\cdot \sum_{i=1}^{r-3}\frac{1}{q^{r+i-3}}=\frac{1}{(q-1)q^{2(n-2k-1)}}\cdot\frac{1}{q^{r-3}} \left( 1-\frac{1}{q^{r-3}}\right).$$
	This implies the first inequality. Write $A_{r}=\frac{1}{(q-1)^{r}q^{r-3}} \left( 1-\frac{1}{q^{r-3}}\right)\cdot\frac{q^{r}-4}{q^{r}}$. Since $A_{r+1}/ A_{r}\leq 1 $ for any $r\geq 4$, we know $A_{r}\leq A_{4}=\frac{1}{(q-1)^{3}}\cdot \frac{q^{2}-2}{q^{4}}\cdot\frac{q^{2}+2}{q^{2}}\leq \frac{3}{16}$.  
	
	For each $1\leq i\leq r-3$, we have  $$\frac{\psi(k,t,t+1, i)}{\psi(k,t,k-i-1, i)}= \frac{{t+1\brack 1}}{{k-i-1\brack t}}\leq \frac{{t+1\brack 1}}{{t+2\brack t}}=\frac{q^{2}-1}{q^{t+2}-1}\leq  \frac{q+1}{q^{2}+q+1}=1-\frac{1}{1+\frac{1}{q}+\frac{1}{q^{2}}}\leq \frac{3}{7}.$$
	This implies the last inequality.
\end{proof}

By the definitions of $a(n,k,t)$ and $h(n,k,t)$, we have 
\begin{equation*}
	\begin{aligned}
		a(n,k,t)&={t+2\brack 1}{n-t-1\brack r-1}-q{t+1\brack 1}{n-t-2\brack r-2},\\
		h(n,k,t)&={n-t\brack r}-q^{r(r+1)}{n-k-1\brack r}+q^{r+1}{t\brack 1}.
	\end{aligned}
\end{equation*}
Recall that $m(n,k,t)=\max\{a(n,k,t), h(n,k,t)\}$.

\begin{Lemma}\label{2609131}
	Let  $k\ge t+2$ and $n\ge 2k+1$. Then the following hold.
	\begin{enumerate}
		\item[\normalfont(i)] If $k\ge 2t+2$, then
		$
		h(n,k,t)>a(n,k,t).
		$
		\item[\normalfont(ii)] If $t+2\le k\le 2t+1$, then
		$a(n,k,t)\ge h(n,k,t)$, 
		with equality if and only if $(k,t)=(3,1)$.
	\end{enumerate}
\end{Lemma}
\begin{proof}
	See \cite[Lemmas~2.5, 2.6(i), 2.7(i), 2.8(i), and~2.9]{Cao-Lv-Wang-Zhou-2022}.
\end{proof}

\begin{Lemma}\label{2608281}
	Let $k\geq t+2$ and $n\geq 2k+1$. Then
	$$m(n,k,t)> \left(1-\frac{1}{3\cdot2^{t}}\right){x\brack 1}{n-t-1\brack k-t-1}. $$
\end{Lemma}
\begin{proof}
	From Lemma \ref{2604282}, we obtain 
	\begin{equation*} \frac{a(n,k,t)}{{t+2\brack 1}{n-t-1\brack r-1}}=1-\frac{q(q^{t+1}-1)(q^{r-1}- 1)}{(q^{t+2}-1)(q^{n-t-1}-1)}>1-\frac{1}{q^{n-k}}\geq 1-\frac{1}{2^{t+3}}\geq 1-\frac{1}{3\cdot2^{t}}. \end{equation*}
	By \cite[Lemma 2.4]{WXY-2024-220}, we get 
	$h(n,k,t)> {r+1\brack 1}{n-t-1\brack r-1}-q{r+1\brack 2}{n-t-2\brack r-2}$. It  follows  that $$\frac{h(n,k,t)}{{r+1\brack 1}{n-t-1\brack r-1}}>1-\frac{1}{(q^{2}-1)q^{n-2k+t-1}}\geq 1-\frac{1}{3\cdot 2^{t}}$$ 
	from Lemma \ref{2604282}.	This finishes  the proof.
\end{proof}

Let $k\geq t+3$, $n\geq 2k+1$, $t+2\leq a\leq k-1$.  
For a real number $0\leq C\leq 315q^{t}/ 512$, write
\begin{equation*}
	F(n,k,t,a;C)
	=
	\frac{315}{512}\cdot
	\frac{1}{q^{(r-1)(n-2k-1)+t(k-a-1)}}
	+
	\frac{C}{q^{(a-t-1)(n-2k+t-1)}}.
\end{equation*}
Then $F(n,k,t,a;C)\leq F(2k+1,k,t,a;C)$. Moreover, the difference of $F(2k+1,k,t,k-1;C)$ and $F(2k+1,k,t,a;C)$ is equal to $(1-q^{-t(k-a-1)})(315/512-Cq^{-t(a-t-1)})\geq 0$.

From Lemma \ref{2604282}, we obtain
\begin{equation}\label{2609171}
	\frac{{r+1\brack1}^{a-t-1}{n-a\brack k-a}}
	{{n-t-1\brack r-1}}=\prod_{i=1}^{a-t-1}
	\frac{(q^{r+1}-1)(q^{r-i}-1)}{(q-1)(q^{n-t-i}-1)}\leq \frac{1}{q^{(a-t-1)(n-2k+t-1)}}.
\end{equation} 
By Lemma \ref{2609153}, we have 
\begin{equation}\label{2609063}
	\frac{\varphi(k,t,a,k,0)}
	{{x\brack1}{n-t-1\brack r-1}}
	+
	\frac{C{x\brack1}{r+1\brack1}^{a-t-1}{n-a\brack k-a}}
	{{x\brack1}{n-t-1\brack r-1}}\leq F(n,k,t,a;C)\leq F(2k+1,k,t,k-1;C).
\end{equation}

\begin{Lemma}\label{2609216}
	Let $k\geq t+3$, $n\geq 2k+1$, $t+2\leq a\leq k-1$ and $(n,k,t,q)\neq (9,4,1,2)$. Set  $r=k-t$,  $U_{1}={r+1\brack1}^{a-t}{n-a\brack k-a}$, $U_{2}={t+2\brack1}{r+1\brack1}^{a-t-1}{n-a\brack k-a}$, $U_{3}={t+4\brack2}{r+1\brack1}^{a-t-2}{n-a\brack k-a}$, 
	\begin{equation*}
		U_{4}={3\brack1}{r+1\brack1}^{a-t-1}{n-a\brack k-a} +q^3{t\brack1}{2\brack1}^2{n-t-2\brack r-2}.
	\end{equation*}
	Then $m(n,k,t)-\Phi_{2}(n,k,t,a,r-3)$ is  greater than $\max\{U_{1}, U_{2},  U_{4}\}$, and $U_{3}$ if $k\geq t+4$.
\end{Lemma}

\begin{proof}
	In this proof,  let $\Psi(n,k,t,a;C)$ denote the LHS of (\ref{2609063}).

	Suppose $t\geq 2$ and $k\geq t+4$. Set $C_{1}={r+1\brack 1}/ {x\brack 1}$, $C_{2}={t+2\brack 1}/ {x\brack 1}$, $C_{3}={t+4\brack 2}/ ({x\brack 1}{r+1\brack 1})$ and $C_{4}= {3\brack 1}/ {x\brack 1}$.  If $k\geq t+4$, then $C_{3}/ q^{t}$ is equal to
	$$\frac{(q^{t+4}-1)(q^{t+3}-1)}{q^{t}(q+1)(q^{x}-1)(q^{r+1}-1)}\leq\frac{1}{q^{r+x-t-6}(q+1)(1-\frac{1}{q^{x}})(1-\frac{1}{q^{r+1}})}\leq \frac{1}{q^{r+x-t-5}}\leq \frac{1}{2}. $$
	By (\ref{2609063}), we have 
	$\Psi(n,k,t,a;C_{i})\leq 347/512$ for any $1\leq i\leq 3$, and  $\Psi(n,k,t,a;C_{4})\leq 323/512$. By  Lemma \ref{2604282}, we have $q^3{t\brack1}{2\brack1}^2{n-t-2\brack r-2}/ \left({x\brack 1}{n-t-1\brack k-t-1}\right)\leq (q+1)^{2}/q^{r+x-2}\leq 9/128$.  
	This together with Lemma \ref{2608317} yields $\Phi_{2}(n,k,t,a,r-3)+U_{i}\leq 11{x\brack 1}{n-t-1\brack r-1}/12$. The desired result follows from Lemma \ref{2608281}.

	Suppose $t\geq 2$ and $k=t+3$. Then $a=t+2$. Moreover, we have $U_{1}\leq U_{2}$ and 
	$(U_{4}-U_{2})/{n-t-2\brack 1}=q^{3}(q+1)(q^{t-1}+q)\geq 0$. It is sufficient to show $\varphi(t+3, t,t+2,t+3,0)+U_{4}<m(n,t+3,t)$. Indeed, we  know $U_4/\left({t+2\brack1}{n-t-1\brack2}\right)$ is equal to $$\frac{q(q+1)^2(q^2-1)}{q^{n-t-1}-1}+\frac{(q^2-1)^2(q^4+1)}{(q^{t+2}-1)(q^{n-t-1}-1)}\le\frac{q(q+1)^2}{(q^2+1)(q^4+1)}+\frac{1}{(q^2+1)^2}<\frac{2}{q^3}+\frac{1}{(q^2+1)^2}.$$
	It is at most $29/100$. This together with Lemmas \ref{2609153} and \ref{2608281} yields the required result.
	
	Suppose $t=1$.  Then $U_{2}\leq U_{1}$, and $U_{3}\leq U_{1}$ if $k\geq 5$. Moreover,  if $k=4$, then
	$U_{1}-U_{4}= q^{4}(q^{2}-1){n-3\brack 1}\geq 0$, and if $n=2k+1$, then by ${k\brack 1}\geq q^{k-2}(q+1)$, we have
	$$\frac{{2k-2\brack k-3}}{{k\brack 1}^{a-3}{2k-a+1\brack k-a}}=\prod_{i=1}^{a-3}\frac{q^{2k-i-1}-1}{{k\brack 1}(q^{k-i-2}-1)}\leq \prod_{i=1}^{a-3}\frac{q^{5}}{(q+1)(q^{2}-1)}\leq q^{2(a-3)},  $$
	which implies $$U_1-U_4\geq {k\brack 1}^{a-3}{2k-a+1\brack k-a}\left({k\brack 1}^{2}-{3\brack 1}{k\brack 1}-q^{2a-3}(q+1)^{2}\right)\geq 0.$$ 
	Note that $\left( \Phi_{2}(n,k,1,a,r-3)+U_{i}\right)/ \left({r+1\brack 1}{n-2\brack r-1}\right)$ is decreasing in $n$  for $1\leq i\leq 4$. By Lemma \ref{2608281}, it is sufficient to show 
	$\left( \Phi_{2}(n,k,1,a,r-3)+U_{1}\right)/ \left({r+1\brack 1}{n-2\brack r-1}\right)\leq  5/6$  at the smallest admissible value of $n$.  For $q\geq 3$,  Lemmas \ref{2608317} and \ref{2609153} together with (\ref{2609171}) give 
	$$\frac{\Phi_{2}(2k+1,k,1,a,r-3)+U_{1}}{{r+1\brack 1}{2k-1\brack r-1}}\leq\frac{4160}{19683} + \frac{3}{16}+\frac{1}{3}\leq \frac{5}{6}.$$  For $q=2$ and $k\geq 7$,   Lemma \ref{2608317} and (\ref{2609063})  give
	$$\frac{\Phi_{2}(2k+1,k,1,a,r-3)+U_{1}}{{r+1\brack 1}{2k-1\brack r-1}}\leq\frac{315}{512} + \frac{7}{64}+\frac{1}{16}\leq \frac{5}{6}.$$
	For $q=2$ and $k=4$,  inequality (\ref{2609063}) gives
	$$\frac{\Phi_{2}(10,4,1,3,0)+U_{1}}{{4\brack 1}{8\brack 2}}\leq F(10,4,1,3;1)=\frac{827}{2048}\leq \frac{5}{6}.$$
	The remaining cases are $n=2k+1$, $t=1$, and $q=2$, with $5\le k\le6$ and $3\le a\le k-1$. 
	By direct calculations, the quotients of $ \Phi_{2}(2k+1,k,1,a,r-3)+U_{1}$ by ${r+1\brack 1}{n-2\brack r-1}$ are at most $5/6$. This finishes the proof.
\end{proof}

\begin{Lemma} \label{2609062}
	Let 
	$n\geq 2t+7$. Then
	$$\varphi(t+3,t,t+2,t+2,0)
	+{t+4\brack2}{n-t-2\brack1}
	<m(n,t+3,t).
	$$
\end{Lemma}

\begin{proof}
	Put $k=t+3$.	For $1\leq j \leq 2$, we get 
	$
	{2\brack j}/{3\brack j}\leq 
	(q+1)/(q^2+q+1)=1-1/(1+1/q+1/q^{2})\leq3/7.
	$
	By $x\geq 4$, we have 
	$
	\varphi(k,t,t+2,t+2,0)
	\leq
	3\varphi(k,t,t+2,k,0)/7
	+4{4\brack1}^{3}/7$.
	Since $x\geq4$ and $n-t-1\geq7$, we know
	$$	\frac{{4\brack1}^{3}}
	{{x\brack1}{n-t-1\brack2}}
	\leq
	\frac{{4\brack1}^{2}}{{7\brack2}}
	=
	\frac{(q+1)(q^4-1)^2}
	{(q^7-1)(q^6-1)}
	\leq \frac{q+1}{q^5}\leq\frac3{32}.
	$$
	This 	together with Lemma \ref{2609153} implies
	$$
	\frac{\varphi(k,t,t+2,t+2,0)}
	{{x\brack1}{n-t-1\brack2}}
	<
	\frac37\cdot\frac{315}{512}
	+\frac47\cdot\frac3{32}
	=\frac{1137}{3584}.
	$$
	Moreover, we know ${t+4\brack2}{n-t-2\brack1}/ \left({x\brack1}{n-t-1\brack2} \right)$ is equal to
	$$
	\frac{(q^{t+4}-1)(q^{t+3}-1)}
	{(q^x-1)(q^{n-t-1}-1)}\leq 		\frac{q^{t+1}}{q^x-1}
	\leq\frac{2^{t+1}}{2^{x}-1}\leq \frac{2^{t+1}}{3(2^{t}+1)}=\frac{2}{3}-\frac{2}{3(2^{t}+1)}\leq \frac{2}{3}-\frac{1}{3\cdot 2^{t}}.
	$$
	The quotient of the LHS 
	by ${x\brack1}{n-t-1\brack2}$ is at most $10579/10752-1/(3\cdot 2^{t})$.  It follows from Lemma \ref{2608281} that the required result holds.
\end{proof}

\begin{Lemma}\label{2609052}
	Let  $k\geq t+2$ and $n\geq2k+1$. Set $r=k-t$.
	If $r\geq 3$, then
	$$
	\begin{aligned}
		\Phi_{1}(n,k,t,t+1,t+1,r-3)
		+{2\brack1}{n-t-1\brack r-1}
		+q^2{t\brack1}{r+1\brack1}{n-t-2\brack r-2}
		<m(n,k,t).
	\end{aligned}
	$$
	If $r=2$, then
	$
	{2\brack1}{n-t-1\brack1}
	+q^2{t\brack1}{3\brack1}<m(n,k,t)
	$.
\end{Lemma}
\begin{proof} 
	
	If $k=t+2$, then $x\geq 3$, and further 
	$$ 	   \frac{(q+1){n-t-1\brack 1}
		+q^{2}{t\brack 1}{3\brack 1}}
	{{3\brack 1}{n-t-1\brack 1}}
	=1-\frac{1}{1+\frac{1}{q}+\frac{1}{q^{2}}}
	+\frac{q^{2}(q^{t}-1)}{q^{n-t-1}-1}
	<\frac{3}{7}+\frac{1}{q^{n-2t-3}}
	\leq\frac{3}{7}+\frac{1}{4}\leq \frac{5}{6}. $$
	In the following, we assume $k\geq t+3$.
	
	Note that $(r-1)(n-2k+t-2)\geq 2t-2$
	and   $ (q^{t+1}-1)/q^{2t-2}\leq q^{2}-1$. 
	It follows from  Lemma \ref{2609153}  that $\varphi(k,t,t+1,k,0)/\left({x\brack 1}{n-t-1\brack r-1} \right)$ is at most 
	$$
	\frac{q^{t+1}-1}{q^{x}-1} \cdot\frac{(q^{2}-1)(q^{3}-1)}{q^{4}(q-1)^{r-1}} \cdot\frac{1}{q^{2t-2}} \leq\frac{(q^{2}-1)^{2}(q^{3}-1)} {q^{4}(q^{4}-1)(q-1)^{2}}
	= \left(\frac{1}{q}+\frac{1}{q^{2}}\right) \left(\frac{1}{q^{2}}+\frac{1}{q(q^{2}+1)}\right) \leq\frac{21}{80}.
	$$
	If $r\geq 4$, then by Lemma \ref{2608317}, we get
	$\sum_{i=1}^{r-3}\psi(k,t,t+1,i){n-k+i\brack i}\leq 
	9{x\brack 1}{n-t-1\brack r-1}/112
	$.
	It  holds for $r=3$ since the sum is empty.
	Thus $\Phi_{1}(n,k,t,t+1,t+1,r-3)/ \left({x\brack1}{n-t-1\brack r-1}\right)\leq 12/35$. 
	Finally, by $k\geq t+3$ and $n\geq 2k+1$, we have
	$$
	\frac{{2\brack 1}{n-t-1\brack r-1}+q^{2}{t\brack 1}{r+1\brack 1}{n-t-2\brack r-2}}
	{{r+1\brack 1}{n-t-1\brack r-1}}
	=\frac{q^{2}-1}{q^{r+1}-1}+q^{2}{t\brack 1}\frac{q^{r-1}-1}{q^{n-t-1}-1}
	\leq \frac{1}
	{q^{2}+1}+\frac{1}{q^{r-1}}\leq\frac{9}{20}.
	$$
	From the discussion above, the quotient of the LHS by
	${x\brack 1}{n-t-1\brack r-1}$ is at most
	$111/140\leq5/6.
	$
	This together with Lemma \ref{2608281} yields the desired result.
\end{proof}

To keep the exposition concise, we evaluate several functions directly at given parameter values in the following lemma.

\begin{Lemma}\label{2609175}
	Let $q=2$. Then the equalities in Table \ref{2609176} hold.
\end{Lemma} 

\begin{table}[htbp]
	\centering
	\caption{Exact comparisons for $q=2$.}
	\label{2609176}
	\renewcommand{\arraystretch}{1.4}
	\setlength{\tabcolsep}{4pt}
	\resizebox{\textwidth}{!}{%
		\begin{tabular}{|l|l|l|}
			\hline
			No. & Upper bounds & $m(n,k,t)$ \\
			\hline
			(i)
			& $\Phi_2(9,4,1,3,0)+{3\brack 1}{4\brack 1}{6\brack1}=28890$
			& $m(9,4,1)=35731$ \\
			\hline
			(ii)
			& $\Phi_2(9,4,1,3,0)+{3\brack 1}{4\brack 1}{6\brack1}+8{2\brack 1}^{2}{6\brack 1}=33426$
			& $m(9,4,1)=35731$ \\
			\hline
			(iii)
			& $\Phi_2(13,6,1,5,1)+\psi(6,1,4,2){9\brack2}=36024308703$
			& $m(13,6,1)=46783294867$ \\
			\hline
			(iv)
			& $\Phi_2(15,7,2,6,1)+\psi(7,2,5,2){10\brack2}=600461767703$
			& $m(15,7,2)=811970486547$ \\
			\hline
			(v)
			& $\Phi_2(17,8,3,7,1)+\psi(8,3,6,2){11\brack2}=9660339410823$
			& $m(17,8,3)=13524074907667$ \\
			\hline
			(vi)
			& $\Phi_2(19,9,4,8,1)+\psi(9,4,7,2){12\brack2}=190977564469863$
			& $m(19,9,4)=224771168900499$ \\
			\hline
			(vii)
			& $\varphi(6,2,5,5,0)+\varphi(6,2,5,5,1){8\brack1}=113073338$
			& $m(13,6,2)=183628627$ \\
			\hline
			(viii)
			& $\varphi(7,3,6,6,0)+\varphi(7,3,6,6,1){9\brack1}=1100780954$
			& $m(15,7,3)=1574407571$ \\
			\hline
			(ix)
			& $\varphi(8,3,7,7,0)+\varphi(8,3,7,7,1){10\brack1}=5104801745526$
			& $m(17,8,3)=13524074907667$ \\
			\hline
			(x)
			& $\varphi(9,4,8,8,0)+\varphi(9,4,8,8,1){11\brack1}=94091538641590$
			& $m(19,9,4)=224771168900499$ \\
			\hline
		\end{tabular}%
	}
\end{table}

\begin{Lemma}\label{260907}
	Let $k\geq t+4$, $n\geq 2k+1$ and $1\leq N\leq k-t-3$.  	If $q\geq 3$ or $n\geq 2k+2$ or $N\geq 2$, then 
	\begin{equation*}
		\Phi_{2}(n,k,t,k-1,N-1) +\psi(k,t,k-N,N){n-k+N\brack N} <m(n,k,t).
	\end{equation*}
\end{Lemma}
\begin{proof}
	
	The  LHS consists of
	$\varphi(k,t,k-1,k,0)$, the sum over $1\leq i\leq N-1$,
	and the terminal term $\psi(k,t,k-N,N){n-k+N\brack N}$. 
	In this proof, let $A$, $B$ and $C$ denote the quotients of these three terms by ${x\brack1}{n-t-1\brack r-1}$, respectively. By Lemma~\ref{2608281}, it is sufficient to show 
	$$A+B+C\leq1-\frac{1}{3\cdot2^t}.$$

	Note  $\psi(k,t,k-N,N)/\psi(k,t,k-N-1,N)=(q^{k-N}-1)/(q^{r-N}-1)$. By (\ref{2609161}), we get
	$$C=\frac{(q^{r-N+1}-1)^{r-N}}{(q-1)^{r-N}}\cdot \frac{q^{t+1}-1}{q^{x}-1}\cdot\prod_{i=1}^{r-N-1}\frac{q^{k-N-i+1}-1}{q^{n-t-i}-1}\cdot \prod_{i=1}^{N-1}\frac{q^{r-i}-1}{q^{N-i+1}-1}.$$
	From  Lemma \ref{2604282} and $(q^{r-i}-1)/(q^{N-i+1}-1)\leq (q^{r-N+1}-1)/(q^{2}-1)$, we obtain
	$$C\leq \frac{q}{(q-1)^{r-1}(q+1)^{N-1}}\cdot\left(\frac{q^{r-N+1}-1}{q^{r-N+1}}\right)^{r-1}\cdot\frac{1}{q^{(r-N-1)(n-2k-1)}}\cdot \frac{1}{q^{(r-N-3)(N-1)}}\cdot\frac{1}{q^{x-t-2}}.$$

	If $q\geq 3$, then  $A\leq 2080/19683$,  $B\leq 3/16$ due to Lemmas \ref{2609153}, \ref{2608317}, and $C\leq q/(q-1)^{r-1}\leq 3/8$. If $n\geq 2k+2$, then $A\leq 1/8$, $B\leq 3/16$ by Lemmas \ref{2609153}, \ref{2608317}, and $C\leq q/ q^{r-N-1}\leq 1/2$.  In either case, we have $A+B+C\leq 5/6$, as desired. 
	
	Suppose $q=2$, $n=2k+1$ and $N\geq 2$.  We first consider $r\geq 6$. By Lemma \ref{2609153}, we know $A\leq 315/512$. Since $i(r-i-1)\geq r+2i-4$ for $1\leq i\leq r-4$, by (\ref{2609163}), we have 
	$$B\leq \sum_{i=1}^{r-4}\frac{1}{2^{i(r-i-1)}}\leq \sum_{i=1}^{r-4}\frac{1}{2^{r+2i-4}}=\frac{1}{3}\cdot\left(1+\frac{1}{2^{r-4}}\right)\cdot\frac{1}{2^{r-4}}\left(1-\frac{1}{2^{r-4}}\right)\leq \frac{5}{64}. $$
	Moreover, by $3^{N-1}\cdot2^{(r-N-3)(N-1)}\geq 3\cdot 2^{N-2}\cdot 2^{r-N-3}\geq 3\cdot 2^{r-5}$,  we have 
	$$C\leq  \frac{2}{3^{N-1}} \cdot \frac{1}{2^{(r-N-3)(N-1)}}\cdot\frac{1}{2^{x-t-2}}\cdot\left(\frac{2^{r-N+1}-1}{2^{r-N+1}}\right)^{r-1}\leq \frac{1}{3\cdot 2^{r-6}}\cdot\frac{1}{2^{x-t-2}} \left(\frac{2^{r-1}-1}{2^{r-1}}\right)^{r-1}. $$
	By a direct calculation for $r=6$ and by bounding the last factor by $1$ for $ r\geq 7$, we obtain $C\leq 2/(7\cdot2^{x-t-2})$ in both cases.  We further conclude
	$$C\leq\frac{2}{7}-\frac{2}{7}\left(1-\frac{1}{2^{x-t-2}}\right)\leq \frac{2}{7}-\frac{1}{3\cdot2^{t}}\left(1-\frac{1}{2^{x-t-2}}\right)\leq \frac{199}{672}-\frac{1}{3\cdot2^{t}}.$$ 
	Then $A+B+C\leq 1-1/(3\cdot 2^{t})$, as required. 
	Next,  assume $r=5$. Then $N=2$. We consider  $t\geq 5$, which implies $x=t+2$. By Lemma \ref{2609153} and (\ref{2609163}), we know 
	$$A\leq \frac{63}{2^{t+2}-1}\cdot \frac{315}{512}\leq \frac{19845}{65024},  \ \  B\leq \frac{7}{8}\cdot \frac{1}{16}=\frac{7}{128}.$$ 
	It follows from $C\leq (2/3)\cdot (15/16)^{4}$ that $A+B+C\leq 95/96$, as desired. For the remaining cases $1\leq t\leq4$,
	Table \ref{2609176} (iii)--(vi)
	verifies the claimed inequality directly.
\end{proof}

In the calculations below, we will use relations between values of $\varphi(k,t,a,b,i)$ at different parameter values. For convenience,  denote its individual summands by
$$\phi(k,t,a,b,i,j)=q^{2j^{2}} {t\brack j}{a-t\brack j}{b-t\brack j} {k-t-i+1\brack1}^{k-t-i-j}.$$
With this notation, we have $ \varphi(k,t,a,b,i)=\sum_{j=0}^{t}\phi(k,t,a,b,i,j)$.

\begin{Lemma}\label{2609099}
	Let $k\geq t+4$, $n=2k+1$ and $q=2$. Then 
	$$\varphi(k,t,k-1,k-1,0)+\varphi(k,t,k-1,k-1,1){k+2\brack 1}<m(n,k,t).$$
\end{Lemma}
\begin{proof}
	For $0\le j\leq \min\{t,r-2\}$, the quotient $\phi(k,t,k-1,k-1,0,j)/\phi(k,t,k-2,k,0,j)$ is 
	$$
	\frac{{r-1\brack j}^{2}}{{r-2\brack j}{r\brack j}}
	=1+\frac{2^{r-j-1}(2^j-1)}
	{(2^{r-j-1}-1)(2^r-1)}
	\le 1+\frac{2^j-1}{2^j+1}
	=2-\frac{2}{2^j+1}
	\le 2-\frac{2}{2^t+1}
	=\frac{2^{t+1}}{2^t+1}.
	$$
	Moreover,  $\varphi(k,t,k-1,k-1,0)$ has at most
	one additional positive term compared with 
	$\varphi(k,t,k-2,k,0)$, namely $\phi(k,t,k-1,k-1,0,r-1)$. This implies 
	$$
	\varphi(k,t,k-1,k-1,0)\leq \frac{2^{t+1}}{2^{t}+1}\varphi(k,t,k-2,k,0)+2^{2(r-1)^{2}}{t\brack r-1}{r+1\brack 1}.
	$$
	For $0\le j\le\min\{t,r-3\}$, we know
	$\phi(k,t,k-1,k-1,1,j)/\phi(k-1,t,k-2,k-1,0,j)$ is equal to $(2^{r-1}-1)/(2^{r-j-1}-1)\leq (2^{r-1}-1)/3$. Moreover, we have 
	$$
	\phi(k,t,k-1,k-1,1,r-2)
	=(2^{r-1}-1)\phi(k-1,t,k-2,k-1,0,r-2).
	$$
	Similarly, a comparison of the positive summands in
	$\varphi(k,t,k-1,k-1,1)$ and
	$\varphi(k-1,t,k-2,k-1,0)$ shows that the former is at most
	$$
	\frac{2^{r-1}-1}{3}\varphi(k-1,t,k-2,k-1,0)+\frac{2^{2(r-2)^2+1}}{3}
	{t\brack r-2}(2^{r-1}-1)^2(2^r-1)
	+2^{2(r-1)^2}{t\brack r-1}.	
	$$
	During this proof, write 
	$$A=\frac{\varphi(k,t,k-1,k-1,0)}{{x\brack 1}{n-t-1\brack r-1}}, \ \ B=\frac{\varphi(k,t,k-1,k-1,1){k+2\brack 1}}{{x\brack 1}{n-t-1\brack r-1}}.$$  By Lemmas \ref{2604282} and \ref{2609153}, we have
	\begin{equation}\label{2609211}
		A\leq \frac{1}{2^{x-r-1}}\cdot \frac{315}{256(2^{t}+1)}+\frac{2^{2(r-1)^{2}}}{2^{x-r-1}}\cdot\frac{{t\brack r-1}}{{2r+t\brack r-1}},
	\end{equation}
	and further  get  that $\varphi(k-1,t,k-2,k-1,0)/\left({x\brack 1}{n-t-1\brack r-1} \right)$ is equal to
	$$\frac{\varphi(k-1,t,k-2,k-1,0)}{{x\brack 1}{n-t-1\brack r-2}}\cdot \frac{2^{r-1}-1}{2^{k+2}-1}\leq \frac{315}{512\cdot 2^{2(r-2)}}\cdot \frac{(2^{r}-1)(2^{r-1}-1)}{(2^{x}-1)(2^{k+2}-1)}.$$
	It follows from $(2^{k+2}-1)/{2r+t\brack r-1}=(2^{r-1}-1)/{2r+t\brack r-2}$, $(2^{k+2}-1)/(2^{x}-1)\leq 2^{r+2}/3$  and Lemma \ref{2604282} that
	\begin{equation}\label{2609212}
		B\leq \frac{105}{256\cdot 2^{x-r-1}}+\frac{2^{2r(r-2)-x+6}}{3}\cdot \frac{{t\brack r-2}}{{2r+t\brack r-2}}+\frac{2^{2(r-1)^{2}+r+2}}{3}\cdot \frac{{t\brack r-1}}{{2r+t\brack r-1}}.
	\end{equation}

	Suppose  $t\le r-3$. Then ${t\brack r-2}={t\brack r-1}=0$. By (\ref{2609211}) and (\ref{2609212}), we have $A+B\le105/128\leq 5/6$.  Suppose $t\ge r-2$. Applying Lemma \ref{2604282}, we have
	${t\brack s}/ {2r+t\brack s}\leq 1/2^{2rs}$ for any  $s\leq 2r+t$, even when $s>t$. This together with (\ref{2609211}) and (\ref{2609212}) yields
	$$
	A+B\leq
	\frac{1}{2^{x-r-1}}
	\left(
	\frac{315}{256(2^t+1)}+\frac{1}{2^{2r-2}}
	+\frac{105}{256}
	\right)+\frac{1}{3\cdot 2^{x-6}}+\frac{1}{3\cdot 2^{r-4}}
	.
	$$
	Indeed, the expression in parentheses is less than $1/2$ whenever $t\geq 4$.
	If $r\geq 6$, then $t\geq 4$, so
	$A+B<1/2+1/6+1/12\leq 5/6$.
	If $4\le r\le5$ and $t\ge r$, then $x=t+2\ge r+2$, giving
	$A+B<1/4+1/3+1/3\leq 11/12$.
	In all these cases, $A+B\leq 1-1/(3\cdot2^t)$, and the desired
	result follows from Lemma \ref{2608281}.
	The remaining cases $(r,t)=(4,2),(4,3),(5,3),(5,4)$ follow
	by direct calculation; see Table \ref{2609176} (vii)--(x).
\end{proof}

\begin{Lemma}\label{2609258}
	Let $t\geq1$, $k\geq t+2$ and $n\geq2k+1$. Set $r=k-t$. Then
	$$
	q{t\brack1}\Phi_1(n,k,t,t,t,r-2) +{n-t-1\brack r-1} +q{r\brack1}{r+1\brack1}{n-t-2\brack r-2}
	<m(n,k,t).
	$$
\end{Lemma}
\begin{proof}
	
	Suppose first  $r=2$.
	Then ${n-t-1\brack1}\geq q^{2}{3\brack1}+1$,  and  $a(n,t+2,t)-{n-t-1\brack 1}$   is equal to 
	$$q{t+1\brack 1}\left( {n-t-1\brack 1}-1\right)\geq q^{3}{t+1\brack 1}{3\brack 1}=q{3\brack 1}\left(q^{3}{t\brack 1}+q^{2}\right)> q{t\brack 1}{3\brack 1}^{2}+q{2\brack 1}{3\brack 1}.$$
	It follows from $\Phi_1(n,k,t,t,t,0)={3\brack 1}^{2}$ that the desired result holds.

	Suppose henceforth that $r\geq3$.    By Lemma \ref{2609152} and ${t\brack 1}\leq q^{(t-1)(r-1)}$, we have 
	$$\frac{q{t\brack 1}{r+1\brack 1}^{r-1}}{{n-t-1\brack r-1}}\leq \frac{(q+1)(q^{t}-1)}{(q-1)^{r-2}q^{t(r-1)+2}}\leq \frac{q+1}{(q-1)^{r-3}q^{r+1}}\leq \frac{3}{16}.$$
	For each $1\leq i\leq r-3$, by ${t+2\brack 2}\geq {3\brack 1}{t\brack 1}$,  we know $q{t\brack 1}\psi(k,t,t,i)/\psi(k,t,k-i-1,i)$ is equal to $q{t\brack 1}/{k-i-1\brack t}$, and it is at most $q{t\brack 1}/{t+2\brack 2}\leq q/{3\brack 1}\leq 2/7$. 
	This together with Lemma \ref{2608317} shows that $\sum_{i=1}^{r-3}q{t\brack 1}\psi(k,t,t,i){n-k+i\brack i}$, divided by ${x\brack 1}{n-t-1\brack r-1}$, is at most $3/56$. Note that
	$$\frac{{t\brack 1}{3\brack 1}^{2}}{{r+1\brack 1}}+{r\brack 1}=\frac{q^{t-1}{3\brack 1}^{2}}{{r+1\brack 1}}+ \frac{{t-1\brack 1}{3\brack 1}^{2}}{{r+1\brack 1}}+{r\brack 1}\leq \frac{q^{t-1}{3\brack 1}^{2}}{{r+1\brack 1}}+q^{t-1}{r\brack 1}, $$
	and further ${3\brack 1}^{2}/ {4\brack 1}+ {r\brack 1}=q^{2}/\left((q+1)(q^{2}+1)\right)+q+1+{r\brack 1}\leq1/(q+1)+q(q^{r-1}+1)$. Then the quotient of $q{t\brack 1}\psi(k,t,t,r-2){n-t-2\brack r-2}+q{r\brack 1}{r+1\brack 1}{n-t-2\brack r-2}$ by ${r+1\brack 1}{n-t-1\brack r-1}$  equals 
	$$q\left(\frac{{t\brack 1}{3\brack 1}^{2}}{{r+1\brack 1}}+{r\brack 1}\right)\cdot\frac{q^{r-1}-1}{q^{n-t-1}-1}\leq \frac{q^{t}(q^{r-1}-1)}{(q+1)(q^{r+k}-1)}+\frac{q^{t+1}(q^{2r-2}-1)}{q^{r+k}-1}\leq \frac{1}{(q+1)q^{r+1}}+\frac{1}{q},$$
	and it is at most $25/48$. Since $1/{r+1\brack1}\leq1/15$, adding the four bounds
	shows that the LHS of the claimed inequality,
	divided by ${x\brack1}{n-t-1\brack r-1}$, is at most
	$3/16+3/56+25/48+1/15\leq5/6$.
	Lemma~\ref{2608281} completes the proof.
\end{proof}

\begin{Lemma}\label{2609257}
	Let  $k\geq t+2$ and $n\geq2k+1$. Set $r=k-t$. Then
	$$
	\begin{aligned}
		{2\brack1}{n-t-1\brack r-1} +\left({r+1\brack1}-{2\brack1}+q^2{t\brack1}\right) {r+1\brack1}{n-t-2\brack r-2}<m(n,k,t).
	\end{aligned}
	$$
	Moreover, for $3\leq s\leq r$,
	$$
	{s\brack1}{n-t-1\brack r-1} +\left({r+1\brack1}-{s\brack1}\right){r+1\brack1}{n-t-2\brack r-2} +q^{r+1}{t\brack1}{n-t-s\brack r-s+1} <m(n,k,t).
	$$
\end{Lemma}
\begin{proof}

	The quotient of the LHS of the first inequality by ${r+1\brack 1}{n-t-1\brack r-1}$ is equal to 
	\begin{equation*}
		\frac{{q+1}}{{r+1\brack 1}}+ \frac{q^{2}\left({r-1\brack 1}+{t\brack 1}\right)(q^{r-1}-1)}{q^{n-t-1}-1}\leq \frac{{q+1}}{{r+1\brack 1}}+ \frac{q^{k}(q^{r-1}-1)}{q^{k+r}-1}\leq \frac{q+1}{q{r\brack 1}}+ \frac{{r-1\brack 1}}{{r\brack 1}}=\frac1q+\frac1{{r\brack1}}
		\leq\frac{5}{6}, 
	\end{equation*}
	where we used ${r-1\brack 1}+{t\brack 1}\leq 2q^{k-3}\leq q^{k-2}$.  Then  the first inequality holds from Lemma \ref{2608281}.
	
	Let $3\leq s\leq r$.
	By Lemma \ref{2604282}, we have  ${s\brack 1}/{r+1\brack 1}\leq 1/ q$ and 
	\begin{equation}\label{2609195}
		\frac{ {r+1\brack 1}{n-t-2\brack r-2}}{{n-t-1\brack r-1}}\leq \frac{1}{(q-1)q^{t}}\leq \frac{1}{(q-1)q}.
	\end{equation}
	Then
	the first two terms in the second inequality, divided by
	${r+1\brack1}{n-t-1\brack r-1}$, have sum at most 
	$$\frac{{s\brack 1}}{{r+1\brack 1}}+\frac{1}{q(q-1)}\cdot\left(1-\frac{{s\brack 1}}{{r+1\brack 1}}\right)=\left(1-\frac{1}{q(q-1)}\right)\frac{{s\brack 1}}{{r+1\brack 1}}+\frac{1}{q(q-1)}\leq\frac{1}{q}+\frac{1}{q^{2}}\leq \frac{3}{4}, $$
	where the first inequality follows by bounding ${s\brack1}/{r+1\brack1}$ by $1/q$
	and simplifying the result to $1/q+1/q^2$. 
	Moreover, by Lemma \ref{2604282}, one can check that ${n-t-s\brack r-s+1}\leq{n-t-3\brack r-2}\leq {n-t-2\brack r-2}/q^{r-2}$.
	It follows from (\ref{2609195}) that the remaining term  divided by ${r+1\brack1}{n-t-1\brack r-1}$  is at most
	$$
	\frac{q^3{t\brack1}}{{r+1\brack1}^2}\cdot \frac{{r+1\brack1}{n-t-2\brack r-2}}{{n-t-1\brack r-1}} \leq \frac{q^3}{(q-1)^2{r+1\brack1}^2} \leq \frac1{q(q-1)^2(q+1)^2} \leq\frac1{18}.
	$$
	In summary, the whole quotient is at most $29/36$, so
	Lemma~\ref{2608281} completes the proof.
\end{proof}

\begin{Lemma}
	Let $q\geq2$ be a prime power. For integers $i\geq0$ and $j\geq3$, we have
	\begin{equation}\label{2609191}
		\frac{q^{2j}{i+1\brack1}{i+2\brack1}}{{i+j+1\brack1}^2{j\brack1}} +\frac{{i+2\brack1}{i+j\brack j-1}{i+j+1\brack1}^{j-2}}{q^{i(2j-1)+j(j-1)}{j\brack1}} \leq 1.
	\end{equation}
	Moreover, if $r\geq3$ and $2\leq j\leq r$, then
	\begin{equation}\label{2609192}
		\frac{q^{(j-1)r+j+1}{r\brack j}}{{r+1\brack1}{2r\brack j-1}} +\frac{{r+1\brack j}{r\brack j-1}{r+1\brack1}^{j-2}}{q^{j(r-j+1)-1}{2r\brack j-1}} \leq 1.
	\end{equation}
\end{Lemma}
\begin{proof}
	If $i=0$, then by   ${j+1\brack 1}\geq q^{j-1}(q+1)$ and ${j\brack 1}\geq q(q+1)$, we have
	$$\frac{q^{2j}{2\brack 1}}{{j+1\brack 1}^{2}{j\brack 1}}+\frac{{2\brack 1}{j+1\brack 1}^{j-2}}{q^{j(j-1)}}\leq \frac{q}{(q+1)^{2}}+\frac{q+1}{q^{2}}\leq \frac{2}{9}+\frac{3}{4}\leq 1.$$
	Suppose $i\geq 1$.  By Lemma \ref{2604282}, the  first term on LHS of  (\ref{2609191}) is at most $q/{j\brack 1}\leq 2/7$, and   
	\begin{equation*}
		\frac{{i+2\brack1}{i+j\brack j-1}{i+j+1\brack1}^{j-2}}{q^{i(2j-1)+j(j-1)}{j\brack1}}= \frac{(q^{i+2}-1)(q^{i+j+1}-1)^{j-2}}{(q-1)^{j-1}q^{i(2j-1)+j(j-1)}}\prod_{\ell=1}^{i} \frac{q^{i+j-\ell+1}-1}{q^{i-\ell+2}-1}\leq \frac{1}{q^{ij}}\cdot\frac{q^{j+1}-1}{q^{2}-1}\cdot
		\frac{q^{j(i-1)}}{(q-1)^{i-1}}.
	\end{equation*}
	It follows from $q/(q^{2}-1)\leq 2/3$ that the LHS of (\ref{2609191}) is at most $2/7+2/3\leq 1$, as desired.
	
	Next, we show the second inequality.  
	Suppose $j=2$. The LHS of (\ref{2609192}) is equal to
	$$
	\frac{q^{r+3}{r-1\brack1}}{{2\brack1}{r+1\brack1}(q^r+1)} +\frac{{r+1\brack1}{r\brack1}}{q^{2r-3}{2\brack1}(q^r+1)} <\frac{q}{q+1} +\frac{q}{q^{r-3}(q+1)(q-1)^2}, 
	$$
	where the last bound is at most $1$ if $r\geq 4$, or $r=3$ and $q\geq3$.
	For $(q,r)=(2,3)$, the LHS of (\ref{2609192}) is $1037/1080\leq 1$. Hence (\ref{2609192}) holds.
	Suppose  $j\geq 3$.  By the identities 
	$$
	\frac{{r-j+1\brack1}{r-j+2\brack1}}{{r+1\brack1}{j\brack1}}=\frac{{r\brack j}}{{r+1\brack j-1}}, \qquad \frac{{r-j+2\brack1}}{{j\brack1}}= \frac{{r+1\brack j}}{{r+1\brack j-1}}, $$
	the LHS of \eqref{2609192} is obtained
	by setting $i=r-j$ into the LHS of
	\eqref{2609191} and multiplying by
	$$
	q^{(j-1)(r-1)}
	\frac{{r+1\brack j-1}}{{2r\brack j-1}}
	=
	\prod_{\ell=0}^{j-2}
	\left(1-\frac{q^{r-1}-1}{q^{2r-\ell}-1}\right)\leq 1.
	$$
	Thus the second inequality holds for $j\geq 3$. This completes the proof.
\end{proof}

\begin{Lemma}\label{2609193}
	Let $k\geq t+3$ and $n\geq2k+1$. Set $r=k-t$. Then
	$$
	q^{r(r+1)}{k-1\brack t-1} +\sum_{i=0}^{t-1}q^{2i(i+1)}{t-1\brack i}{r+1\brack i+1}{r\brack i+1}{r+1\brack1}^{r-i-1} <m(n,k,t).
	$$
\end{Lemma}
\begin{proof}
	
	Taking $(a,b,c)=(r,t-1,t-1)$ in Lemma \ref{2609151}, we obtain
	$$
	q^{r(r+1)}{k-1\brack t-1} =\sum_{i=0}^{t-1} q^{r(r+1)+i^2}{t-1\brack i}{r\brack i}.
	$$
	During this proof, for $i\geq 0$, write $$A_{i}=q^{r(r+1)+i^2}{t-1\brack i}{r\brack i}, \ \ B_{i}=q^{2i(i+1)}{t-1\brack i}{r+1\brack i+1}{r\brack i+1}{r+1\brack1}^{r-i-1}.$$
	
	Suppose that $k\geq 2t+2$. 
	Counting the members of $\mathcal H(X,M)$ containing $X$ and meeting $M$
	in dimension  $t+1$, we know $h(n,k,t)>q^{r(r-1)}{r+1\brack1}{n-k-1\brack r-1}$.   Taking $(a,b,c)=(t-1,r+1,r-1)$ in Lemma \ref{2609151}, by $n\geq 2k+1$, we obtain
	$$h(n,k,t)>q^{r(r-1)}{r+1\brack1}{k\brack r-1}=\sum_{i=0}^{t-1}q^{r(r-1)+i(i+2)}{r+1\brack 1}{t-1\brack i}{r+1\brack r-i-1}.$$
	For each $0\leq i\leq t-1$, by the identities 
	$$\frac{{r\brack i}}{{r+1\brack r-i-1}}=\frac{{r\brack i}}{{r\brack i+1}}\cdot \frac{{r\brack i+1}}{{r+1\brack i+2}}=\frac{{i+1\brack 1}{i+2\brack 1}}{{r-i\brack 1}{r+1\brack 1}},\ \ \frac{{r+1\brack i+1}}{{r+1\brack r-i-1}}=\frac{{r+1\brack i+1}}{{r+1\brack i+2}}=\frac{{i+2\brack 1}}{{r-i\brack 1}}, $$
	we have  the following equality
	$$\frac{A_{i}+B_{i}}{q^{r(r-1)+i(i+2)}{r+1\brack 1}{t-1\brack i}{r+1\brack r-i-1}}=\frac{q^{2(r-i)}{i+1\brack 1}{i+2\brack 1}}{{r+1\brack 1}^{2}{r-i\brack 1}}+\frac{{i+2\brack 1}{r\brack i+1}{r+1\brack 1}^{r-i-2}}{q^{r^{2}-i^{2}-r}{r-i\brack 1}}.$$
	Applying \eqref{2609191} with $j=r-i$ and summing over $i$,   we obtain $\sum_{i=0}^{t-1}(A_{i}+B_{i})<h(n,k,t)$.

	Suppose that  $k\leq 2t+1$. 
	Counting the members of $\mathcal A(Z)$ whose intersection with $Z$ has dimension $t+1$,
	and taking $(a,b,c)=(t-1,2r,r-1)$ in Lemma \ref{2609151}, we have
	\begin{equation*}
		a(n,k,t)  > q^{r-1}{t+2\brack1}{t+2r-1\brack r-1}
		= \sum_{i=0}^{r-1}q^{r-1+i(r+i+1)}{t+2\brack1}{t-1\brack i}{2r\brack r-i-1}.
	\end{equation*}
	For convenience, let $C_{i}$ be the summand with index $i$ on the
	RHS of the preceding inequality.  
	
	For each $0\leq i\leq r-2$, we have ${t-1\brack i}>0$, and further
	$$
	\frac{{t+2\brack1}(A_i+B_i)}{{r+1\brack1}C_i}
	=
	\frac{q^{r^2-i(r+1)+1}{r\brack i}}{{r+1\brack1}{2r\brack r-i-1}}
	+\frac{{r+1\brack i+1}{r\brack i+1}{r+1\brack1}^{r-i-2}}
	{q^{(r-i)(i+1)-1}{2r\brack r-i-1}}.
	$$
	Setting $j=r-i$ in (\ref{2609192}), we get $(A_i+B_i)/C_i\leq {r+1\brack 1}/ {t+2\brack 1}\leq 1$ for any $0\leq i\leq r-2$.

	The remaining positive terms on the left involve at most
	$A_{r-1},A_{r}$,  and $B_{r-1}$. It is sufficient to show $A_{r-1}+A_{r}+B_{r-1}\leq C_{r-1}+C_{r-2}\left({t+2\brack 1}-{r+1\brack 1}\right)/ {t+2\brack 1}$. 
	Note that 
	$$
	A_{r-1}+A_r
	=q^{r(2r-1)+1}
	\left({r\brack1}{t-1\brack r-1}
	+q^{2r-1}{t-1\brack r}\right)
	\leq q^{r(2r-1)+1}{r\brack1}{t\brack r},
	$$
	where  we used $q^{2r-1}\leq q^{r}{r\brack 1}$. 
	By  ${r\brack1}{t\brack r}={t\brack1}{t-1\brack r-1}$
	and $q^{r+1}{t\brack1}+{r+1\brack1}={k+1\brack1}$, we obtain
	\begin{equation}\label{2609197}
		A_{r-1}+A_r+B_{r-1}
		\leq q^{2r(r-1)}{k+1\brack1}{t-1\brack r-1}.
	\end{equation}
	Since $q^{r-1}{t+2\brack1}={k+1\brack1}-{r-1\brack1}$, we have
	\begin{equation}\label{2609198}
		C_{r-1}
		=q^{2r(r-1)}{k+1\brack1}{t-1\brack r-1}-q^{2r(r-1)}{r-1\brack1}{t-1\brack r-1}.
	\end{equation}
	Using ${r-1\brack1}{t-1\brack r-1}
	={t-r+1\brack1}{t-1\brack r-2}$ and $q^{r+1}{t-r+1\brack1}={t+2\brack1}-{r+1\brack1}
	$,
	we obtain 
	$$
	q^{2r(r-1)}{r-1\brack1}{t-1\brack r-1}
	=\frac{q^{2r-1}{t-r+1\brack1}}{{t+2\brack1}{2r\brack1}}C_{r-2}
	=\frac{q^{r-2}}{{2r\brack1}}
	\left(1-\frac{{r+1\brack1}}{{t+2\brack1}}\right)C_{r-2}\leq \left(1-\frac{{r+1\brack1}}{{t+2\brack1}}\right)C_{r-2}.
	$$	
	It follows from  (\ref{2609197}) and (\ref{2609198}) that  $A_{r-1}+A_{r}+B_{r-1}\leq C_{r-1}+C_{r-2}\left({t+2\brack 1}-{r+1\brack 1}\right)/ {t+2\brack 1}$.  This completes the proof.
\end{proof}

\end{document}